\documentclass[a4paper]{amsart}
\usepackage{amsfonts}
\usepackage{amsmath}
\usepackage{amssymb}
\usepackage{graphicx}

\providecommand{\U}[1]{\protect\rule{.1in}{.1in}}
\newtheorem{theorem}{Theorem}
\theoremstyle{plain}

\newtheorem{conjecture}{Conjecture}
\newtheorem{corollary}{Corollary}

\newtheorem{definition}{Definition}

\newtheorem{lemma}{Lemma}

\newtheorem{proposition}{Proposition}
\newtheorem{remark}{Remark}

\numberwithin{equation}{section}
\input{tcilatex}

\begin{document}
\title[Aubry-Mather theorem for elastic chains]{The Aubry-Mather theorem for
driven generalized elastic chains}
\author{Sini\v{s}a Slijep\v{c}evi\'{c}}
\address{Department of Mathematics, Bijeni\v{c}ka 30, Zagreb, Croatia}
\email{slijepce@math.hr}
\urladdr{}
\date{3 May 2013}

\begin{abstract}
We consider uniformly (DC)\ or periodically (AC)\ driven generalized
infinite elastic chains (a generalized Frenkel-Kontorova model) with
gradient dynamics. We first show that the union of supports of all the
invariant measures, denoted by $\mathcal{A}$, projects injectively to a
dynamical system on a 2-dimensional cylinder. We also prove existence of
ergodic invariant measures supported on a set of rotationaly ordered
configurations with an arbitrary (rational or irrational) rotation number.
This shows that the Aubry-Mather structure of ground states persists if an
arbitrary AC or DC\ force is applied. The set $\mathcal{A}$ attracts almost
surely (in probability) configurations with bounded spacing. In the DC case, 
$\mathcal{A}$ consists entirely of equilibria and uniformly sliding
solutions. The key tool is a new weak Lyapunov function on the space of
translationally invariant probability measures on the state space, which
counts intersections.
\end{abstract}

\maketitle

\section{Introduction}

We consider the equation%
\begin{equation}
\frac{d}{dt}%
u_{j}(t)=-V_{2}(u_{j-1}(t),u_{j}(t))-V_{1}(u_{j}(t),u_{j+1}(t))+F(t)\text{%
,\quad }j\in \mathbb{Z}  \label{r:FK}
\end{equation}%
such that $V:\mathbb{R}^{2}\rightarrow \mathbb{R}$ is $C^{2}$, periodic\ ($%
V(u+1,v+1)=V(u,v)$), and satisfies the twist condition $V_{12}(u,v)\leq
-\delta <0$ for some fixed $\delta \in \mathbb{R}$.

We assume that $F$ is either constant (the DC\ case) or 1-periodic in $t$
(the AC\ case). In the AC case, we further assume that $V$, $F$ are real
analytic.

\subsection{\textbf{The background and motivation. }}

The case $F=0$ is the standard setting of the well-known Aubry-Mather theory
(\cite{Baesens98}, \cite{Katok05}, Section 9.3), as the solutions of $%
du/dt=0 $ correspond to orbits of an area-preserving twist map $f_{V}$ on a
cylinder whose generating function is $V$. The Aubry-Mather theory describes
minimizing configurations/orbits and measures, and so simultaneously
characterizes physically important "ground states" of (\ref{r:FK}) and
describes "KAM-circles" and what is left of them for $f_{V}$.

Our goal is to contribute towards the description of the dynamics of (\ref%
{r:FK}) on the entire state-space, allowing a fairly general family of
initial conditions. We are, however, inspired by results on asymptotics of
one-dimensional reaction-diffusion equations (\cite{Fiedler89}, \cite{Joly10}%
), their braid dynamics (\cite{Ghrist09}), stability of synchronization
phenomena (\cite{Gallay11});\ in the Hamiltonian dynamics setting the\
Mather's measure-theoretical approach to minimizing orbits and its
extensions (\cite{Mather89}); PDE extensions of Aubry-Mather theory (\cite%
{Llave09});\ results on spatio-temporal entropy of extended PDE's (\cite%
{Mielke09}, \cite{Zelik03}); all related to what is actually known on the
behavior of (\ref{r:FK}) either numerically, physically (\cite{Braun:04}) or
analytically (\cite{Baesens98}, \cite{Baesens05}, \cite{Hu05}, \cite{Qin10}, 
\cite{Qin11}). Before stating precisely the actual results, we will try to
put them in a wider context of these topics.

In the "standard" case $V(u,v)=(u-v)^{2}/2-W(u)$, where $W(u+1)=W(u)$ and $W:%
\mathbb{R}\rightarrow \mathbb{R}$ is $C^{2}$, the equation (\ref{r:FK}) is a
spatial discretization of the scalar reaction-diffusion equation%
\begin{equation*}
u_{t}=u_{xx}+W^{^{\prime }}(u)\text{,}
\end{equation*}%
and it is closely related to a more general reaction-diffusion equation%
\begin{equation}
u_{t}=u_{xx}+f(x,u,u_{x}),  \label{r:rfd}
\end{equation}%
where $f$ is $1$-periodic in the first coordinate. Both (\ref{r:FK})\ and\ (%
\ref{r:rfd})\ are scalar and order-preserving (or monotone, or cooperative,
see \cite{Smith:96}). The well-known Fiedler and Mallet-Paret Poincar\'{e}%
-Bendixson theorem \cite{Fiedler89} showed that the asymptotics of (\ref%
{r:rfd}) on any bounded domain is 2-dimensional, which is a basis of a
relatively complete understanding of dynamics of (\ref{r:rfd}) on bounded
domains (\cite{Joly10} and references therein). The key tool they applied is
the intersection-counting function $(u,v)\rightarrow z(u-v)$, which is a
discrete Lyapunov function. That means that $t\mapsto z(u(t)-v(t))$ is
non-increasing along two solutions of (\ref{r:rfd}).\ Furthermore, if $%
u(t_{0})-v(t_{0})$ has a singular zero (we also say that $u(t_{0})$ and $%
v(t_{0})$ intersect non-transversally), $z(u(t)-v(t))$ strictly drops at $%
t=t_{0}$.

These ideas have not in our knowledge been applied to the equation (\ref%
{r:rfd}) on the entire $\mathbb{R}$, considered as an "extended system"
(i.e. without assuming $u(x)\rightarrow 0$ as $x\rightarrow \infty $). The
reason is that the intersection-counting function $z$ is typically not
finite when $u,v$ are defined on the entire $\mathbb{R}$, thus all
bounded-domain tools fail. We believe that the right approach in the
extended-domain case is to consider the extension of the
intersection-counting function $z$ to the space of translationally invariant
measures on the state space.

In this paper we pursue this idea for a technically simpler discrete-space
system (\ref{r:FK}), and consider a function%
\begin{equation*}
(\mu ^{1},\mu ^{2})\longmapsto Z(\mu ^{1},\mu ^{2})
\end{equation*}%
which counts the average number of intersections of configurations with
respect to two translationally invariant probability measures on the space
of configurations $(u_{i})_{i\in \mathbb{Z}}$. We will show that $t\mapsto
Z(\mu ^{1}(t),\mu ^{2}(t))$ is a weak Lyapunov function (for details see\
Section \ref{s:measure}), where $\mu ^{1}(t),\mu ^{2}(t)$ are evolutions of
measures with respect to (\ref{r:FK}). We believe these ideas can also be
applied also to (\ref{r:rfd}) and give fuller insight on the dynamics of (%
\ref{r:rfd}) on unbounded domain. In particular, we conjecture that all our
results (with appropriate rephrasing)\ hold also for (\ref{r:rfd}).

Van der Berg, Ghrist, Vandervorst and W\'{o}jcik studied in several papers
topological "braid" properties of intersections of solutions of (\ref{r:rfd}%
) on bounded domains and its spatial discretization (\cite{Ghrist09} and
references therein). Our approach may lead to extension of their ideas to
unbounded domains.

It has been noted that synchronized solutions of the extended
reaction-diffusion equation in more spatial dimensions ($\geq 2$) are often
stable and locally attracting (see e.g. \cite{Gallay11} for recent results
and references). We argue that it is possible to study this phenomenon also
in the spatial dimension 1. As compared to \cite{Gallay:01}, we also allow
unbounded solutions $u_{i}(t)$ (by identifying $u$ and $u+1$ and considering
the dynamics in an appropriate quotient space), and also study
non-autonomous systems (AC\ forcing). We eventually reduce our dynamics to a
2d flow (DC case)\ or map (AC\ case) with non-trivial dynamics (the
spatially $1d$ dynamics in \cite{Gallay11} reduces to a trivial $1d$ flow).
A heuristic explanation for stability of $1d$ synchronization is now as
follows:\ the synchronized solutions are the solutions in the support of
invariant measures~$\mu $ such that the self-intersecting function $Z(\mu
,\mu )$ is equal to zero (we consider measures invariant both for spatial
translations and time evolution). As the function $t\mapsto Z(\mu (t),\mu
(t))$ is non-decreasing, such solutions are expected to be at least Lyapunov
stable. We develop a rigorous approach to this idea.

The Aubry-Mather theory of ground states of (\ref{r:FK})\ in the case $F=0$
has been extended to more dimensional lattices and PDEs (\cite{Llave09} and
references therein). In our knowledge, generalizations typically deal with
existence of ground states with an arbitrary generalized rotation number\
(rational or irrational). In the PDE setting, this leads to construction of
various solutions of stationary PDEs. In our view, it is somewhat unexpected
that the Aubry-Mather stationary structure seems to be preserved in a not
entirely obvious way for an evolution equation (\ref{r:FK}) with arbitrarily
strong, and even time-periodic forcing.

Finally, let us come back to the actual dynamics of (\ref{r:FK}). The
Frenkel-Kontorova model is one of the most studied physical models, as it
reasonably realistically models a variety of physical phenomena, and is
perhaps a simplest model with sufficiently rich structure allowing lots of
"metastable states", various phase transitions such as "pinning-depinning",
etc. (\cite{Braun:04} and references therein). It has been long noted in the
physics literature that the dynamics of (\ref{r:FK}) seems to be attracted
to low-dimensional invariant manifolds, or to spatially synchronize. Baesens
and MacKay proved that in the DC case spatially periodic initial conditions
converge either to equilibria or to uniformly sliding solutions\ (\cite%
{Baesens98}, \cite{Baesens05};\ see below for details). Hu, Bambi and Qin
extended some of those results to show existence and properties of spatially
periodic, ordered solutions in the AC\ case (\cite{Hu05}), and\ Qin studied
existence and properties of uniformly sliding solutions with both rational
and irrational rotation number in the DC\ case (\cite{Qin10}, \cite{Qin11}%
).\ 

Here we attempt to give a broader perspective on the dynamics of (\ref{r:FK}%
) on a large state-space, and generalize most of the mentioned results.
Firstly, we consider dynamics on the space of configurations of bounded
spacing (configurations $(u_{i})_{i\in \mathbb{Z}}$ such that $\sup_{i\in 
\mathbb{Z}}|u_{i+1}-u_{i}|<\infty $), which is more general than the
configurations of bounded width (i.e. such that $\sup_{i\in \mathbb{Z}%
}|u_{i}-i\rho |<\infty $ for some $\rho \in \mathbb{R}$) rigorously studied
so far. This is important, as it includes configurations connecting two
configurations with different rotation number, and various other relevant
situations. We then show existence of a 2d invariant set of configurations
in both AC and DC\ case which in a physically meaningful way attracts
configurations with bounded spacing. Finally, we show existence of ordered
configurations with an arbitrary (rational or irrational)\ rotation number.

\subsection{\textbf{The statement of results.}}

The state space for the dynamics of (\ref{r:FK}) of configurations of
bounded spacing will be denoted by $X$, and defined precisely as follows:%
\begin{eqnarray*}
K_{n} &=&\{u\in \mathbb{R}^{Z}\text{, such that }\sup_{i\in \mathbb{Z}%
}|u_{i+1}-u_{i}|\leq n\}\text{,} \\
X &=&\tbigcup_{n=1}^{\infty }K_{n},
\end{eqnarray*}%
with the induced product topology. We will use the quotient sets $\mathcal{X}%
=X/R$, $\mathcal{K}_{n}=K_{n}/R$, where $Ru=u+1$ is a translation commuting
with (\ref{r:FK}) (for details see Section \ref{s:measure}).

The system (\ref{r:FK}) generates a continuous global semiflow and a local
flow on $\mathcal{X}$ (see Section 4), which we denote by $\varphi $ (or $%
\varphi ^{t}$ for a given $t$). Of key importance will be the translation $%
T:X\rightarrow X$ defined naturally as $(Tu)_{i}=u_{i-1}$, commuting with (%
\ref{r:FK}) and $\varphi $.

The Borel probability measures on $\mathcal{X}$ will be called \textit{%
measures with bounded spacing}. We will in particular study measures with
bounded spacing, invariant with respect to both the translation $T$ and the
semilflow $\varphi ^{t}$, called $(\varphi ,T)$-invariant. In the DC\ case,
the measures will be invariant with respect to $\varphi ^{t}$ for all $t\in 
\mathbb{R}$, and in the AC\ case with respect to $t\in \mathbb{Z}$. Denote
the sets of such measures on $\mathcal{K}_{n},\mathcal{X}$ by $\mathcal{E}(%
\mathcal{K}_{n})$, $\mathcal{E}(\mathcal{X})$. Let%
\begin{equation}
\mathcal{A}=\cup _{\mu \in \mathcal{E}(\mathcal{X})}\text{supp}(\mu )\text{,}
\label{r:ea}
\end{equation}%
where supp$(\mu )$ denotes the support of $\mu $. Our first main result is
the following:

\begin{theorem}
\label{t:main1}There is a continuous injective projection $\pi :\mathcal{A}%
\rightarrow \mathbb{T}^{1}\times \mathbb{R}$.
\end{theorem}

The injective projection will be given by%
\begin{equation}
\pi (u)=(u_{0},u_{1}-u_{0})\text{.}  \label{r:pi}
\end{equation}

As the projection $\pi $ is injective, we can pull the translation $T$, and
the semiflow $\varphi ^{t}$ in the DC\ case, or the map $\varphi ^{1}$ to
the 2-dimensional set $\pi (X)$. We introduce the following notation:%
\begin{equation}
\begin{array}{l}
\widetilde{\mathcal{X}}=\pi (\mathcal{X}), \\ 
\widetilde{\mathcal{K}}_{n}=\pi (\mathcal{K}_{n}), \\ 
\widetilde{T}=\pi \circ T\circ \pi ^{-1}, \\ 
\widetilde{\varphi }^{t}=\pi \circ \varphi ^{t}\circ \pi ^{-1}\text{.}%
\end{array}
\label{r:DEF}
\end{equation}

\begin{theorem}
\label{t:main2}For each $n\in \mathbb{N}$, the map $\widetilde{T}$ is a
homeomorphism of $\widetilde{\mathcal{K}}_{n}$, in the DC$\ $case $%
\widetilde{\varphi }^{t}$ is a continuous flow on $\widetilde{\mathcal{K}}%
_{n}$, in the AC\ case $\widetilde{\varphi }^{1}$ is a homeomorphism on $%
\widetilde{\mathcal{K}}_{n}$, and $\widetilde{T}$ and $\widetilde{\varphi }%
^{t}$ (respectively $\widetilde{\varphi }^{1}$) commute.
\end{theorem}

We propose to call the pair $(\widetilde{T}$,$\widetilde{\varphi })$ in the
DC case the \textit{characteristic flow-map}, and the pair $(\widetilde{T}$,$%
\widetilde{\varphi }^{1})$ in the AC case the \textit{characteristic maps}.

We give a detailed example in Section \ref{s:example} for the case $F=0$,
and relate our main results to the Aubry-Mather theory. In that case in
particular, $\widetilde{\mathcal{X}}$ is the entire cylinder $\mathbb{T}%
^{1}\times \mathbb{R}$ if $V$ is close to integrable, $\widetilde{\varphi }%
^{t}$ is the identity, and $\widetilde{T}$ is diffeomorphic to an
area-preserving twist map whose generating function is $V$ (\cite{Katok05}).

We do not know whether in general the maps $\widetilde{T}$ and $\widetilde{%
\varphi }^{t}$ can be naturally extended to the entire cylinder $\mathbb{T}%
^{1}\times \mathbb{R}$. Also unlike in the equilibrium case $F=0$, we do not
a-priori know $\widetilde{T}$, $\widetilde{\varphi }^{t}$. We give some
conjectures in Section \ref{s:characteristic}. We can, however, show that $%
\widetilde{\mathcal{X}}\subset \mathbb{T}^{1}\times \mathbb{R}$ is a"rich"
set for all $F$, as it contains analogues to Aubry-Mather sets of minimizing
configurations/orbits with an arbitrary \textit{rotation number} (called
also \textit{mean spacing}) found in the case $F=0$, and to Mather's
minimizing measures (\cite{Mather89}). (We postpone definitions of rotation
number, a $(\varphi ,T)$-ergodic measure, a rotationally ordered
configuration, and totally ordered orbit to Sections \ref{s:measure}, \ref%
{s:characteristic}.)

\begin{theorem}
\label{t:main3} Given any $\rho \in \mathbb{R}$, there exists a $(\varphi
,T) $-ergodic measure on $\mathcal{X}$ supported on a set of rotationally
ordered configurations with the rotation number $\rho $.
\end{theorem}

In the DC case, we can completely describe the set of all $(\varphi ,T)$%
-invariant and ergodic measures:

\begin{theorem}
\label{t:main4}\ Assume $F$ is DC. Then every $(\varphi ,T)$-ergodic measure
on $\mathcal{X}$ is supported either on a set of equilibria, or on a single
totally ordered periodic orbit.
\end{theorem}

The totally ordered periodic orbits in Theorem \ref{t:main4} are in the
literature known as uniformly sliding states (\cite{Baesens98}, \cite%
{Baesens05}, \cite{Qin10}, \cite{Qin11}). We show in Section \ref{s:sliding}
that the known results on uniformly sliding states follow as a direct
corollary from our considerations.

Finally, we address asymptotics of the dynamics (\ref{r:FK}) relative to the
set $\mathcal{A}$ defined by (\ref{r:ea}).

\subsection{\textbf{The spatio-temporal attractor. }}

We propose a notion of attractiveness which is weaker than the usual notions
of attractor or union of $\omega $-limit sets, but for extended systems such
as (\ref{r:FK}) often in practice (physical systems or numerical
simulations)\ indistinguishable from them. Imagine a physical situation in
which we observe dynamics of (\ref{r:FK}) on $\mathbb{Z}$ (or a very large
subset of $\mathbb{Z}$), and that no position in space is preferred.
Consider also that an initial condition $u(0)$ emerges as a realization of
some spatially homogenous random process.\ For example, we can choose $u(0)$
as a double-infinite sequence so that $u_{i}$ (or $u_{i+1}-u_{i}$) is chosen
randomly and independently in each step $i\in \mathbb{Z}$ from the same
bounded measurable subset of $\mathbb{R}$. We can study that as analyzing
the dynamics (\ref{r:FK}) with respect to $\mu $-a.e initial condition,
where $\mu $ is a $T$-invariant probability measure on $\mathcal{X}$. We
propose a notion of spatio-temporal attractiveness, and say that a set $%
\mathcal{A}$ is a spatio-temporal attractor, if given any $T$-invariant
probability measure $\mu $, then for $\mu $-a.e. initial condition $u$ and
with density of times $t$ equal to $1$, $\varphi ^{t}(u)$ is arbitrarily
close to $\mathcal{A}$. More precisely, given any $S>0$, we define the Borel
probability measure $P_{\mu ,S}$ on $\mathcal{X}\times \lbrack 0,S]$ as the
product measure $\mu \times \lambda _{S}$, where $\lambda _{S}=(1/S)\lambda
|_{[0,S]}$ is the normed Lebesgue measure $\lambda $ restricted to $[0,S]$,
and define the attractor as follows:

\begin{definition}
A set $\mathcal{A}$ is a spatio-temporal attractor for (\ref{r:FK}) on $%
\mathcal{X}$, if given any $T$-invariant probability measure $\mu \ $on $%
\mathcal{X}$ and any open neighbourhood $U$ of $\mathcal{A}$, 
\begin{equation}
\lim_{S\rightarrow \infty }P_{\mu ,S}(\{(u,t),\text{ }\varphi ^{t}(u)\in
U\})=1\text{.}  \label{d:attractor}
\end{equation}
\end{definition}

We finally in Section \ref{s:attractor} we show that the set $\mathcal{A}$
described in Theorems \ref{t:main1}-\ref{t:main4}, is indeed a
spatio-temporal attractor.

\begin{theorem}
\label{t:main5}The set $\mathcal{A}$ defined with (\ref{r:ea}) is the
spatio-temporal attractor for (\ref{r:FK}) on $\mathcal{X}$.
\end{theorem}

Lastly, we comment on the spatio-temporal entropy of (\ref{r:FK}). A
corollary of Theorem \ref{t:main4} is that the spatio-temporal entropy of (%
\ref{r:FK}) on $\mathcal{X}$ in the DC\ case is zero (see e.g. \cite%
{Mielke09}, \cite{Zelik03} for definitions). A proof follows from the
variational principle for metric and topological spatio-temporal entropy and
our description of all spatio-temporal invariant measures;\ we intend to
report details separately. This complements well the existing results for
PDE\ systems such as (\ref{r:rfd}). Specifically, Zelik \cite{Zelik03}
showed that extended gradient systems (the case $F=0$ of (\ref{r:FK})) have
spatio-temporal entropy zero, Mielke and Zelik \cite{Mielke09} constructed $%
1d$ non-autonomous PDE\ examples with positive entropy, and Turaev and Zelik 
\cite{Turaev10} proved positive entropy for a $1d$ complex system.

\subsection{\textbf{The structure of the paper. }}

In Section \ref{s:example} we as an example summarize the known facts on the
dynamics (\ref{r:FK})\ in the "extended gradient" case $F=0$, and relate
them to our results. In Sections \ref{s:zeroset}, \ref{s:intersections} we
develop tools on the dynamics of the zero-set of a non-autonomous system of
infinitely many linear ODEs, and as a corollary on dynamics of intersections
of a cooperative system of infinitely many ODEs. Here we extend some of the
known results (\cite{Baesens05}, \cite{Smillie84}, \cite{Smith:96})\ to
systems of infinitely many ODEs, less smooth then e.g. in \cite{Smillie84},
and on a general space of solutions (the Fr\'{e}chet space of configurations
of sub-exponential growth at infinity). As such, our results could be of
independent interest. We postpone most of the technicalities to Appendices
1,2, dealing in particular with some exceptional, "degenerate" situations
due to the infinite number of equations. In Section \ref{s:measure} we
introduce the semiflow on the space of translationally invariant measures on 
$\mathcal{X}$, and show that the average-number-of-intersections-counting
function $Z(\mu _{1},\mu _{2})$ is non-increasing and strictly decreasing if 
$u,v$ in the supports of $\mu _{1},\mu _{2}$ have a non-transversal
intersection (or, using the alternative terminology of the paper, $u-v$ has
a singular zero). Once all the tools are in place, the proofs of Theorems %
\ref{t:main1}--\ref{t:main5} are relatively straightforward and given in
Sections \ref{s:characteristic}--\ref{s:attractor}, including a more
detailed discussion of uniformly sliding states.

\section{Example:\ dynamics without driving and Aubry-Mather theory \label%
{s:example}}

In this section we relate our results to the known results in the case of no
driving $F=0$. Our assumptions on the function $V$ are the standard setting
of the Aubry-Mather theory. Let us first consider the equilibria $du/dt=0$
of (\ref{r:FK}), denoted by $\mathcal{E}\subset \mathbb{R}^{\mathbb{Z}}/R$,
where we again identify $u$ and $Ru=u+1$ and consider the quotient set. It
is straightforward to deduce from the conditions on $V$ that the projection $%
\pi :\mathcal{E\rightarrow }\mathbb{T}^{1}\times \mathbb{R}$ is a
homeomorphism (again considering the induced product topology on $\mathcal{E}
$). Also the map $\widetilde{T}:\mathbb{T}^{1}\times \mathbb{R\rightarrow T}%
^{1}\times \mathbb{R}$ is $C^{2}$, diffeomorphic to the twist map generated
by $V$ (\cite{Katok05}, Section 9.3). The diffeomorphism is given by the
discrete Legendre transform, which is for example\ identity in the
"standard" case%
\begin{equation}
V(u,v)=(v-u)^{2}/2+W(u),  \label{r:standard}
\end{equation}%
where $W(u+1)=W(u)\,$and $W$ is a $C^{2}$ function.

It is however not a-priori clear that all the $(\varphi ,T)$-invariant
measures of (\ref{r:FK}) with $F=0$ are supported on $\mathcal{E}$. The
system (\ref{r:FK}) is an example of an extended gradient system, i.e. a
system which has a formal Lyapunov function (in our case the formal sum $%
\sum_{i=-\infty }^{\infty }V(u_{i},u_{i+1})$), and is a gradient system when
reduced to a finite domain.\ It has been shown that extended gradient
systems can have more complex dynamics than what is by LaSalle principle
possible for gradient systems (\cite{Gallay:01}, \cite{Gallay:12}, \cite%
{Slijepcevic00}), including possible existence of non-stationary orbits in $%
\omega $-limit sets.

In \cite{Slijepcevic00} we however showed the following:

\begin{theorem}
All $(\varphi ,T)$-invariant measures of (\ref{r:FK})\ on $\mathcal{X}$ in
the case $F=0$ are supported on $\mathcal{E}$.
\end{theorem}

The key tool is an observation that the function $\mu \mapsto \int
V(u_{0},u_{1})d\mu $ is a Lyapunov function on the space of $T$-invariant
probability measures $\mathcal{X}$, so LaSalle principle must hold for the
evolution of $T$-invariant measures.

We conclude that Theorems \ref{t:main1}-\ref{t:main3} do hold in the case $%
F=0$, where the measures in Theorem \ref{t:main3} are for example Mather's
minimizing measures (\cite{Bangert88}, \cite{Mather89}).

Let us briefly discuss the question whether and when $\pi :\mathcal{%
A\rightarrow }\mathbb{T}^{1}\times \mathbb{R}$ is a homeomorphism, where $%
\mathcal{A}$ is the support of all $(\varphi ,T)$-invariant measures of
bounded spacing as defined in (\ref{r:ea}). One can see that $\pi $ is a
homeomorphism if and only if $\mathcal{E}\subset \mathcal{X}$. This is true
because $\widetilde{T}$ is diffeomorphic to an area-preserving map, so we
can pull back a Lebesgue measure and construct $(\varphi ,T)$-invariant
measures supported on any closed subset of $\mathcal{E\cap X}$. In general, $%
\mathcal{E}\subset \mathcal{X}$ is true if and only if the area-preserving
twist map $f_{V}$ generated by $V$ is close enough to integrable. For
example, in the standard case (\ref{r:standard}), $\mathcal{E}\subset 
\mathcal{X}$ holds if and only if $f_{V}$ has homotopically non-trivial
invariant circles,\ otherwise one can construct "accelerating" invariant
measures of $f_{V}$;\ i.e. there are $T$-invariant measures in $\mathcal{E}$
which do not have bounded spacing (\cite{Slijepcevic99}).

Finally, let us consider asymptotics. The following slightly stronger
version of Theorem \ref{t:main5} in the case $F=0$ was shown in \cite%
{Slijepcevic00} and extended in \cite{Gallay:01}, \cite{Gallay:12} for much
more general systems:

\begin{theorem}
\label{t:seven}Consider the dynamics (\ref{r:FK})\ with $F=0$ on $\mathcal{X}
$.

(i)\ If $\mu $ is a $T$-invariant Borel probability measure on $\mathcal{X}$%
, then the $\omega $-limit set for $\mu $-a.e. $u\in \mathcal{X}$ consists
of equilibria.

(ii)\ Let $u\in \mathcal{X}$ and let $U$ be any open neighborhood of $%
\mathcal{E}$. Then the set of times $t\in \lbrack 0,\infty )$ such that $%
u(t)\in U$ has Banach density $1$.
\end{theorem}

We can rephrase Theorem \ref{t:seven}, and say that the $\omega $-limit set
consists of equilibria either almost surely with respect to space, or with
respect to time. One can not omit the "almost surely" part of the statement
and still obtain a $2d$ union of limit sets either in Theorem \ref{t:main5}
or \ref{t:seven}, as it is most likely infinite-dimensional, analogously to
the continuous space case (\cite{Miranville08}).

\section{The dynamics of a zero-set \label{s:zeroset}}

Here we first consider in general terms properties of the zero-set of a
solution of a linear, cooperative non-autonomous system of infinitely many
ordinary differential equations. We show that the number of zeroes in a
bounded domain can not increase and is strictly decreasing at singular
zeroes. We then extend these ideas to systems of infinitely many equations,
and establish a "zero-balance equation". Although some of the ideas
presented here are essentially for finite dimensional systems already
contained in \cite{Smillie84}, \cite{Smith:96}, our presentation follows
more closely Angenent, Fiedler and Mallet-Paret \cite{Angenent88a}, \cite%
{Angenent88}, \cite{Fiedler89} who developed it in the PDE setting. In
particular, we use the PDE-terms "regular" and "singular zeroes", rather
than "transversal" and "non-transversal intersections" of $v=0$, as we will
in a way consider the dynamics of the zero-set.

Let us consider a system of equations%
\begin{equation}
\frac{d}{dt}w_{j}(t)=a_{j}(t)w_{j-1}+b_{j}(t)w_{j+1}+c_{j}(t)w_{j}\text{%
,\quad }j\in \mathbb{Z}  \label{r:system}
\end{equation}%
We assume that for some $T_{0}<T_{1}$, the following holds:%
\begin{equation}
\begin{array}{l}
a_{j},b_{j},c_{j}\text{ continuous, uniformly bounded in }j,t\in \lbrack
T_{0},T_{1}]\text{,} \\ 
a_{j},b_{j}\geq \delta >0\text{ for all }j,t\in \lbrack T_{0},T_{1}]\text{.}%
\end{array}
\label{r:conditions}
\end{equation}

The condition $a_{j},b_{j}\geq \delta >0$ means that the system (\ref%
{r:system}) is cooperative (\cite{Smillie84}, \cite{Smith:96}). We can
consider (\ref{r:system}) to be a discrete analogue of parabolic
differential equations considered by Angenent (in real analytic case in \cite%
{Angenent88a}, in $C^{2}$ case in \cite{Angenent88}). We consider (\ref%
{r:system})\ on the space $Y$ of configurations of sub-exponential growth.
We define it as the set $Y\subset \mathbb{R}^{\mathbb{Z}}$ of $%
w=(w_{i})_{i\in \mathbb{Z}}$ such that%
\begin{equation*}
\lim \sup_{j\rightarrow \infty }\frac{\log |w_{j}|}{|j|}=0.
\end{equation*}%
Indeed, $Y$ is a Fr\'{e}chet space generated by the family of norms 
\begin{equation}
||u||_{n,\infty }=\sup_{j\in \mathbb{Z}}\exp (-|j|/n)|u|_{j},  \label{r:norm}
\end{equation}%
$n$ a positive integer.

We first give definitions. We say that $(j,t)\subset \mathbb{Z}\times
\lbrack T_{0},T_{1}]$ is a zero of a $w:[T_{0},T_{1}]\rightarrow Y$, if the
graph of $i\mapsto w_{i}(t)$ intersects $y=0$ in the interval $[j,j+1)$
(i.e. 
\begin{equation*}
w_{j}(t)+(w_{j+1}(t)-w_{j}(t))x=0
\end{equation*}%
has a solution for $x\in \lbrack 0,1)$). We say that $(j,t)$ is a singular
zero, if $w_{j}(t)=0$, and either of the following three cases holds:%
\begin{eqnarray*}
w_{j+1}(t) &=&0\text{, or} \\
w_{j-1}(t) &=&0\text{, or} \\
w_{j+1}(t)w_{j-1}(t) &>&0
\end{eqnarray*}%
(these are the discrete space analogues of $w_{x}(t)=0$);\ otherwise $(j,t)$
is a regular zero.\ We denote by $Z,S,R\subset \mathbb{Z}\times \lbrack
T_{0},T_{1}]$ the sets of zeroes, singular zeroes and regular zeroes of a
given $w(t)$. A singular zero has a degree $k$, if $k$ is the maximal number
of consecutive $j$ such that $(j,t)$ is a zero. Assume now $%
(j+1,t),...,(j+k,t)\in S$ is a regular zero of a finite degree $k$, $%
(j,t),(j+k+1,t)\not\in Z$. We distinguish two types of singular zeroes:%
\begin{eqnarray*}
\text{Type I, if }w_{j}(t)w_{j+k+1}(t) &<&0\text{,} \\
\text{Type II, if }w_{j}(t)w_{j+k+1}(t) &>&0\text{.}
\end{eqnarray*}%
Note that for singular zeroes of Type I, the degree is $k\geq 2$, and for
Type II, $k\geq 1$ (see Figure 1 in the Appendix 1). A zero can also be of
degree $\infty $, that is $w(t)\in Y$ such that for all $i\geq i_{0}$, $%
w_{i}(t)=0$, or such that for all $i\leq i_{0}$, $w_{i}(t)=0$, but $w$ not
identically $0$. When discussing zeroes, we will need to consider this,
intuitively degenerate and quite unlikely case separately with some care.

Given $w\in Y$, we define $z_{j}(w)$ to be $1$ if $w$ has a zero at $j$,
otherwise $0$. Let $z_{m,n}:Y\rightarrow \mathbb{N\cup }\{0\}$ be the
zero-counting function

\begin{equation*}
z_{m,n}(w)=\sum_{j=m}^{n-1}z_{j}(w)\text{,}
\end{equation*}%
where $m<n$ are integers.

We say that a solution of (\ref{r:system}) is $C^{1}$ if it is $C^{1}$ in
each coordinate. We can now state the discrete analogue of Angenent \cite%
{Angenent88}, Theorem D (i.e. for free boundary conditions not intersecting
zero):

\begin{lemma}
\label{t:zeroset}Assume $w:\mathbb{[}T_{0},T_{1}]\rightarrow Y$ is a $C^{1}$
solution of (\ref{r:system}) satisfying (\ref{r:conditions}), $m<n$
integers. Assume also $w_{j}(t)\not=0$ for $(j,t)\in \{m,n\}\times \lbrack
T_{0},T_{1}]$. Then

(i)\ The number of zeroes $z_{m,n}(w(t))$ is non-increasing.

(ii) If $w(t_{0})$ has a singular zero between $m$ and $n$, then $%
z_{m,n}(w(t))$ is strictly decreasing at $t=t_{0}$.

(iii)\ If $w(t_{0})$ has a singular zero between $m$ and $n$, then there
exists $\delta _{0}>0$ so that for all $0<|\delta |\leq $ $\delta _{0}$, $%
w(t_{0}+\delta )$ has no singular zeroes between $m$ and $n$.
\end{lemma}

We postpone the proof (as well as all the other proofs in this section) to
Appendix 1 as it is rather technical, and the approach is not related to the
rest of the paper.

We would now like to make general statements on values of $z_{m,n}$ for
arbitrary time and space intervals, including those when a zero is
"crossing" a boundary $\{m,n\}$ of a segment, and the assumptions of Theorem %
\ref{t:zeroset} do not hold.

We will associate to a solution $w:[T_{0},T_{1}]\rightarrow Y$ of (\ref%
{r:system})\ a family of functions $c_{i}(w;s,t)$, $i\in \mathbb{Z}$, $0\leq
s\leq t$, counting the number of times a zero enters minus the number of
times it leaves the segment $i\geq 0$ during the time interval $(s,t]$ (a
precise definition will be given below). In addition, we define $%
d_{i}(w;s,t) $, $i\in \mathbb{Z}$, $0\leq s\leq t$ as the number of zeroes
"disappearing" at the position $i$ in the time interval $(s,t]$. We will
omit the argument $w$ when it is clear from the context.

As the definitions of $c_{i},d_{i}$ are intuitively straightforward but
technically somewhat subtle, we give it axiomatically, and then prove their
existence in Appendix 1.

\begin{description}
\item[(A1)] \textbf{The zero-balance equation:} for all integers $m<n$ and
all $T_{0}\leq s<t\leq T_{1}$ the following holds:%
\begin{equation}
z_{m,n}(w(t))-z_{m,n}(w(s))=c_{m}(s,t)-c_{n}(s,t)-\sum_{j=m}^{n-1}d_{j}(s,t).
\label{r:balance}
\end{equation}

\item[(A2)] \textbf{Translation invariance}:\ $c_{j},d_{j}$ are invariant
for the translation $T$, that is%
\begin{eqnarray*}
c_{j}(Tw;s,t) &=&c_{j+1}(w;s,t)\text{,} \\
d_{j}(Tw;s,t) &=&d_{j+1}(w;s,t).
\end{eqnarray*}

\item[(A3)] \textbf{Regularity of }$c$: If $w_{j}(t)\not=0$ for each $t\in $ 
$[t_{0},t_{1}]$, then $c_{j}(t_{0},t_{1})=0$.

\item[(A4)] \textbf{Cardinality of }$\dot{d}$: If $%
\{(m,t_{0}),...,(m+k-1,t_{0})\}$ is a zero of degree $k$, then there is $%
\delta _{0}>0$ so that for all $0<\delta \leq \delta _{0}$,%
\begin{equation*}
\sum_{j=m}^{m+k-1}d_{j}(t_{0}-\delta ,t_{0}+\delta )=\left\{ 
\begin{array}{ll}
k-1 & \text{zero of Type I,} \\ 
k & \text{zero of Type II.}%
\end{array}%
\right.
\end{equation*}

If there are no singular zeros at $j$ for $t$ $\in \lbrack t_{0},t_{1}]$,
then $d_{j}(t_{0},t_{1})=0$.
\end{description}

We again postpone the proof of the following to the Appendix:

\begin{lemma}
\label{t:zbe}We can associate to each $C^{1}$ solution $w:[T_{0},T_{1}]%
\rightarrow Y$ of (\ref{r:system}) satisfying (\ref{r:conditions}) a family
of functions $c_{j}(w;s,t)\in \mathbb{Z}$\thinspace\ and $d_{j}(w;s,t)\in 
\mathbb{\{}0,1,..\}$\thinspace\ satisfying (A1)-(A4).
\end{lemma}

Now recall the partial order on $\mathbb{R}^{\mathbb{Z}}$:%
\begin{equation*}
w\geq v\text{ if for all }j\text{, }w_{j}\geq v_{j}\text{.}
\end{equation*}%
The following well-known monotonicity property (\cite{Baesens05}, \cite%
{Smith:96}) can be understood as a special case of the discussion above;\
for the sake of completeness a short proof is also in Appendix 1.

\begin{lemma}
\label{p:monotonicity}Assume $w:[T_{0},T_{1}]\rightarrow Y$ is a solution of
(\ref{r:system}) satisfying (\ref{r:conditions}). Then if $w(0)\geq 0$, then
for all $t>0$, $w(t)\geq 0$.
\end{lemma}

\section{The dynamics of intersections \label{s:intersections}}

In this section we consider the dynamics of intersections $w=u^{1}-u^{2}$ of
two solutions $u^{1},u^{2}$ of (\ref{r:FK}) and of derivatives $w=du/dt$ of
a solution $u$ of (\ref{r:FK}). We first show that the results of the
previous section apply to $u^{1}-u^{2}$ in both the AC and DC case, and to $%
du/dt$ in the DC\ case. We then establish local and global existence of a
continuous semiflow generated by (\ref{r:FK}). Finally in a series of four
lemmas, we establish the key tool:\ that the singular zeroes of $u^{1}-u^{2}$
and $du/dt$ are persist for small perturbations in the product topology.
Here we treat the cases of singular zeroes of finite and infinite degree
separately, as the later case is technically more different.

\subsection{Existence of semiflow.}

Recall the definition of $K_{n}$ as the set of all $u\in \mathbb{R}^{Z}$
such that $|u_{j+1}-u_{j}|\leq n$, and fix $n\in \mathbb{N}$.

\begin{proposition}
\label{p:existence}Assume $u^{1},u^{2},u:[T_{0},T_{1}]\rightarrow K_{n}\,\ $%
are solutions of (\ref{r:FK}). Then $w=u^{2}-u^{1}$ in all cases, and $%
w=du/dt$ in the DC\ case, are solutions of (\ref{r:system}) with some $%
a_{i},b_{i},c_{i}$ satisfying (\ref{r:conditions}).
\end{proposition}

\begin{proof}
If $w=u^{2}-u^{1}$, we define%
\begin{eqnarray*}
a_{i} &=&\frac{V_{2}(u_{i-1}^{2},u_{i}^{1})-V_{2}(u_{i-1}^{1},u_{i}^{1})}{%
u_{i-1}^{1}-u_{i-1}^{2}}, \\
b_{i} &=&\frac{V_{1}(u_{i}^{1},u_{i+1}^{2})-V_{1}(u_{i}^{1},u_{i+1}^{1})}{%
u_{i+1}^{1}-u_{i+1}^{2}}, \\
c_{i} &=&\frac{%
V_{2}(u_{i-1}^{2},u_{i}^{2})-V_{2}(u_{i-1}^{2},u_{i}^{1})+V_{1}(u_{i}^{2},u_{i+1}^{2})-V_{1}(u_{i}^{1},u_{i+1}^{2})%
}{u_{i}^{1}-u_{i}^{2}}
\end{eqnarray*}%
(naturally extended to $a_{i}=-V_{12}$, $b_{i}=-V_{12}$, $%
c_{i}=-V_{11}-V_{22}\,\ $\ in the cases $u_{i-1}^{1}=u_{i-1}^{2}$, $%
u_{i+1}^{1}=u_{i+1}^{2}$, $u_{i}^{1}=u_{i}^{2}$ respectively). Then by the
twist condition and the mean value theorem, $a_{i},b_{i}\geq \delta >0$, and
by periodicity of $V$ and a simple compactness argument, $%
a_{i}(t),b_{i}(t),c_{i}(t)$ are bounded, thus (\ref{r:conditions}) holds. As 
$u^{1}(t),u^{2}(t)\in K_{n}$, the difference $j\mapsto
|u_{j}^{1}(t)-u_{j}^{2}(t)|$ grows at most linearly in $|j|\,$, thus $%
u^{1}(t)-u^{2}(t)\in Y$.

Now let $w=du/dt$ in the DC\ case. Then by differentiating (\ref{r:FK}) we
see that (\ref{r:system}), (\ref{r:conditions}) hold with%
\begin{eqnarray*}
a_{i} &=&-V_{12}(u_{i-1},u_{i}), \\
b_{i} &=&-V_{12}(u_{i},u_{i+1}), \\
c_{i} &=&-V_{22}(u_{i-1},u_{i})-V_{11}(u_{i},u_{i+1})\text{.}
\end{eqnarray*}
\end{proof}

We now establish that the equation (\ref{r:FK})\ generates a smooth semiflow
on $\mathcal{K}_{n}=K_{n}/R$, denoted by $\varphi ^{t}$ and $\varphi $.

\begin{lemma}
\label{l:existence}The equation (\ref{r:FK}) generates a continuous semiflow
on $\mathcal{K}_{n}$, and a continuous local flow on $\mathcal{X}$.

For each $i\in \mathbb{Z}$, the solution $t\mapsto u_{i}(t)$ is $C^{2}$;\
and if $F,V$ are real analytic, $t\mapsto u_{i}(t)$ is also real analytic,
all on all open sets of all $t$ for which the solution $u(t)\in \mathcal{X}$.
\end{lemma}

\begin{proof}
It is easy to show by applying the well-known results on existence of
solutions on Banach spaces, that for $u(0)\in \mathcal{K}_{n}$ the solution
of (\ref{r:FK})\ exists locally, and $t\mapsto u(t)-u(0)$ is continuous in $%
l_{\infty }(\mathbb{Z})$. Thus (\ref{r:FK}) generates a local flow on $%
\mathcal{X}$.

We can show that the solution depends continuously on initial conditions in
the product topology, by for example establishing local existence of
solutions in any of the norms $||.||_{n,\infty }$ introduced in the previous
section (note that these norms induce the product topology on $\mathcal{K}%
_{n}$ and $\mathcal{X}$). Finally, given a positive integer $n$, the
property $u\in \mathcal{K}_{n}$ can be written as%
\begin{equation*}
Tu+n\geq u\geq Tu-n,
\end{equation*}%
which is by Lemma \ref{p:monotonicity} and Proposition \ref{p:existence}
applied to $w(t)=(Tu+n)(t)-u(t)$, $w(t)=u(t)-(Tu-n)(t)$ invariant forward in
time. We conclude that if $u(0)\in \mathcal{K}_{n}$, then $u(t)\in \mathcal{K%
}_{n}$ for all $t\geq 0$.

Regularity follows from (\ref{r:FK}). In the real analytic case we use real
analyticity of $F,V$ and check inductively that the derivatives satisfy the
characterization of real analyticity (\cite{Krantz02}, Proposition 1.2.12).
\end{proof}

We will often fix an integer $n$ and write $\mathcal{K}$ omitting its
superscript.

Generally the results to follow regarding the difference $u^{1}-u^{2}$ will
hold in both AC\ and DC\ cases, and the results regarding the derivative $%
du/dt$ in the DC case only.

\subsection{\textbf{The AC\ and DC\ cases. }}

We first show that singular zeroes of a finite degree persist for small
perturbations in the product topology. Without loss of generality we
consider zeroes at $t_{0}=0$.

\begin{lemma}
\label{l:finite}Assume $u^{1}(0),u^{2}(0)\in \mathcal{K}$ such that $%
u^{1}(0)-u^{2}(0)$ has a singular zero of a finite degree $k$ at $%
(m+1,t_{0}),...,(m+k,t_{0})$. Then for all sufficiently small $\varepsilon
>0\,$there exist open neighborhoods $U^{1}$ and $U^{2}$ of $u^{1}(0)$, $%
u^{2}(0)$, such that for all $v^{1}(0),v^{2}(0)\in U^{1}\times U^{2},$%
\begin{equation}
\sum_{j=m+1}^{m+k}d_{j}(v^{1}-v^{2};-\varepsilon ,+\varepsilon )\geq \frac{k%
}{2}\text{.}  \label{r:finite}
\end{equation}
\end{lemma}

\begin{proof}
By continuity and Lemma \ref{t:zeroset}, (iii) there exists $\delta
_{0}>0\,\ $so that $u^{1}(t)-u^{2}(t)$ is not zero on $\{m,m+k+1\}\times
\lbrack -\delta _{0},\delta _{0}]\,$and so that $u^{1}(t)-u^{2}(t)$ has no
other singular zeroes between $0$ and $k+1$ for $t\in \lbrack -\delta
_{0},\delta _{0}]$. By Lemma \ref{t:zbe}, (A3) and then (A1) and (A4), we
see that for all $0<\varepsilon \leq $ $\delta _{0}$ 
\begin{eqnarray}
z_{m+1,m+k}(u^{1}(\varepsilon )-u^{2}(\varepsilon ))
&-&z_{m+1,m+k}(u^{1}(-\varepsilon )-u^{2}(-\varepsilon ))  \notag \\
&=&-\left\{ 
\begin{array}{ll}
k-1 & \text{zero of Type I} \\ 
k & \text{zero of Type II}%
\end{array}%
\right. \leq -\frac{k}{2}.  \label{r:k}
\end{eqnarray}

Now, given any $0<\varepsilon \leq $ $\delta _{0}$, as all the zeroes of $%
u^{1}(\pm \varepsilon )-u^{2}(\pm \varepsilon )$ are regular at $%
i=m+1,...,m+k$, we can by continuity find neighborhoods $U^{1}$ and $U^{2}$
of $u^{1}(0)$ and $u^{2}(0)$ so that for any $v^{1}(0),v^{2}(0)\in
U^{1}\times U^{2}$,%
\begin{eqnarray*}
z_{m+1,m+k}(u^{1}(\varepsilon )-u^{2}(\varepsilon ))
&=&z_{m+1,m+k}(v^{1}(\varepsilon )-v^{2}(\varepsilon )), \\
z_{m+1,m+k}(u^{1}(-\varepsilon )-u^{2}(-\varepsilon ))
&=&z_{m+1,m+k}(v^{1}(-\varepsilon )-v^{2}(-\varepsilon )),
\end{eqnarray*}%
and so that $v^{1}(t)-v^{2}(t)$ is not zero on $\{m,m+k+1\}\times \lbrack
-\varepsilon ,\varepsilon ]$. The claim now follows from applying Lemma \ref%
{t:zbe}, (A3) and then (A1) to $v^{2}-v^{1}$.
\end{proof}

The zeroes of infinite degree however do not necessarily persist for small
perturbations in the product topology, as a zero in a small neighborhood can
"escape" towards $\pm \infty $. We can however show the following:

\begin{lemma}
\label{l:infinite}Assume $u(0),v(0)\in \mathcal{K}$ such that $u(0)-v(0)$
has a zero of degree $\infty $.\ Then for each $\varepsilon >0$ there exist $%
|t_{0}|,|t_{1}|<\varepsilon $ so that either $u(t_{0})-v(t_{1})$ or $%
u(t_{0})-u(t_{1})$, $u(t_{0})\not=u(t_{1})$ has a singular zero of a finite
degree.
\end{lemma}

The proof of Lemma \ref{l:infinite} is simple but quite involved as it
requires analysis of a number of cases with respect to the sign of
derivatives $du^{1}(0)/dt$, $du^{2}(0)/dt$. We give it in Appendix 2.

\subsection{\textbf{The DC\ case.}}

We now state analogues for the Lemmas above for $du/dt$ instead of $%
u^{1}-u^{2}$. Assume in the next two lemmas that $F$ is DC.

\begin{lemma}
\label{l:der1}Assume $u(0)\in \mathcal{K}$, such that $du/dt$ has a singular
zero of a finite degree $k$ at $(m+1,0),...,(m+k,0)$. Then for all
sufficiently small $\varepsilon >0\,$there exist a neighborhood $U$ $\ $of $%
u(0)$ such that for all $v(0)\in U,$%
\begin{equation*}
\sum_{j=m+1}^{m+k}d_{j}(dv/dt;t_{0}-\varepsilon ,t_{0}+\varepsilon )\geq 
\frac{k}{2}.
\end{equation*}
\end{lemma}

The proof is analogous to the proof of Lemma \ref{l:finite}. Similarly, we
can show (for the proof also see Appendix 2):

\begin{lemma}
\label{l:der2}Assume $u(0)\in \mathcal{K}$ such that $du/dt$ has a zero of
degree $\infty $.\ Then for each $\varepsilon >0$ there exist $%
|t_{0}|,|t_{1}|<\varepsilon $ so that $u(t_{0})-u(t_{1})$, $%
u(t_{0})\not=u(t_{1})$ has a singular zero of a finite degree.
\end{lemma}

\section{The average number of intersections and zeroes \label{s:measure}}

This section is the core of the paper. We apply here the results of the
previous section to solutions of (\ref{r:FK}), and define average number of
self-intersections and derivative zeroes of a translationally invariant
measure on $\mathcal{X}$, and average number of intersections of two
probability measures on $\mathcal{X}$. We then show that these functions
(denoted by $Z$) are weak Lyapunov, that is non-decreasing with respect to
the evolution of (\ref{r:FK}). Furthermore, if there is a singular zero in
the support of these measures, we show that $Z$ is strictly decreasing.
Finally, these properties are in some sense continuous in the weak topology
of measures.

These ideas are measure-theoretical equivalents of the properties of the
function $z$ counting intersections of two solutions of (\ref{r:FK}) and (%
\ref{r:rfd}) (\cite{Fiedler89}).

The intersection-counting function $(u,v)\longmapsto z_{i,j}(u-v)$ is not
invariant for the transformation $R$, thus not well defined on the quotient
sets $\mathcal{X}$, $\mathcal{K}$. To fix that, we introduce two-argument
functions $z_{i},z_{i,j}:\mathcal{X}\times \mathcal{X\rightarrow }\mathbb{N}%
\cup \{0\}$ with 
\begin{eqnarray*}
z_{i}(u,v) &=&\sum_{r\in \mathbb{Z}}z_{i}(u-v+r), \\
z_{i,j}(u,v) &=&\sum_{r\in \mathbb{Z}}z_{i,j}(u-v+r).
\end{eqnarray*}

First note that $z_{i},z_{i,j}$ are finite. Indeed, if $u,v\in \mathcal{K}%
_{n}$, then it is easy to see that $z_{i}(u,v)\leq 2n+1$, thus $%
z_{i,j}(u,v)\leq (2n+1)(j-i)$. Also $z_{i},z_{i,j}$ are invariant for the
transformation $R$ in either argument, thus well defined on $\mathcal{X}%
\times \mathcal{X}$.

Our key object will be the space of $T$-invariant Borel probability measures 
$\mathcal{M}(\mathcal{K})$ on $\mathcal{K}$, and 
\begin{equation*}
\mathcal{M}(\mathcal{X})=\bigcup_{k=1}^{\infty }\mathcal{M}(\mathcal{K}_{n})%
\text{.}
\end{equation*}%
We will always consider the weak topology on $\mathcal{M}(\mathcal{K})$ and $%
\mathcal{M}(\mathcal{X})$ with respect to the product topology on $\mathcal{X%
}$. Note that $\mathcal{K}$ is homeomorphic to a product of compact sets,
thus compact. As $\mathcal{K}$ is compact and $T$ is continuous, $\mathcal{M}%
(\mathcal{K})$ is closed, and by the Tychonoff-Alaoglu theorem $\mathcal{M}%
(K)$ is compact.

Before proceeding, let us describe the space $\mathcal{M}(\mathcal{K})$. It
is clearly non-empty (containing e.g. the Dirac measure $\delta _{0}$, where 
$0$ is $u\equiv 0$), and actually very large. Most invariant measures of
area-preserving twist maps generated by any generating function $V$ are
naturally embedded in $\mathcal{M}(\mathcal{K}_{n})$ for sufficiently large $%
n$ (see Section \ref{s:example}). We can also embed translationally
invariant measures on $S^{\mathbb{Z}}$, where $S$ is a bounded measurable
subset of $\mathbb{R}$, in different ways in $\mathcal{M}(\mathcal{K})$.

The semiflow $\varphi ^{t}$ now naturally induces a continuous semiflow $%
\varphi _{\ast }^{t}$ on $\mathcal{M}(\mathcal{K}_{n})$. The flow $\varphi
^{t}$ on $\mathcal{X}$ also induces a local flow $\varphi _{\ast }^{t}$ on $%
\mathcal{M}(\mathcal{X})$, as one can easily see that one can find $\delta
(n)>0$ so that for all $u(0)\in \mathcal{K}_{n}$, $u(t)\in \mathcal{K}_{n+1}$
for $t\in \lbrack -\delta (n),0]$ (e.g. by compactness argument). We will
often use the notation $\mu (t)=\varphi _{\ast }^{t}\mu (0)$.

Given $\mu ^{1},\mu ^{2},\mu \in \mathcal{M}(\mathcal{K}_{n})$, we define
the intersection-counting functions%
\begin{eqnarray*}
Z(\mu ^{1},\mu ^{2}) &=&\int \int z_{0}(u,v)d\mu ^{1}(u)d\mu ^{2}(v), \\
Z(\mu ) &=&Z(\mu ,\mu ), \\
\widetilde{Z}(\mu ) &=&\int z_{0}\left( \frac{du}{dt}\right) d\mu (u)\text{,}
\end{eqnarray*}%
where $du/dt$ is given by the right-hand side of (\ref{r:FK}). By definition,%
\begin{eqnarray*}
0 &\leq &Z(\mu ^{1},\mu ^{2})\leq 2n+1, \\
0 &\leq &\widetilde{Z}(\mu )\leq 1\text{.}
\end{eqnarray*}%
The function $Z(\mu ^{1},\mu ^{2})$ measures the average number of
intersections with respect to two measures;\ the function $Z(\mu )$ measures
the average number of "self-intersections" with respect to $\mu $, and $%
\widetilde{Z}(\mu )$ measures the average number of critical points $%
du_{0}/dt=0$ with respect to (\ref{r:FK}). First note that the definitions
above are invariant of the spatial location as $\mu ^{1},\mu ^{2},\mu $ are
by definition $T$-invariant, thus the following holds for any integers $n<m$:%
\begin{eqnarray*}
Z(\mu ^{1},\mu ^{2}) &=&\int z_{n}(u,v)d\mu ^{1}(u)d\mu ^{2}(v) \\
&=&\frac{1}{m-n}\int z_{n,m}(u,v)d\mu ^{1}(u)d\mu ^{2}(v),
\end{eqnarray*}%
analogously for $Z(\mu ),\widetilde{Z}(\mu )$.

We will see that $Z,\widetilde{Z}$, are "almost"\ Lyapunov functions in the
sense of LaSalle, thus they severely constrain the dynamics of $\varphi
_{\ast }$. Generally the properties below will hold for functions $Z$ in
both AC and\ DC cases, and for the function $\widetilde{Z}$ in the DC\ case
only.

\begin{proposition}
\label{p:decreasing1}If $\mu ^{1}(0),\mu ^{2}(0),\mu (0)\in \mathcal{M}(%
\mathcal{K})$, then the functions $t\mapsto Z(\mu ^{1}(t),\mu ^{2}(t)),$ $%
t\mapsto Z(\mu (t))$ are non-increasing.
\end{proposition}

\begin{proof}
We write shortly $\mu ^{1},\mu ^{2}$ instead of $\mu ^{1}(0),\mu ^{2}(0)$.
Let 
\begin{equation}
Z(t):=Z(\mu ^{1}(t),\mu ^{2}(t)).  \label{r:shorthand}
\end{equation}%
We now have for $0\leq s<t$%
\begin{equation*}
Z(t)-Z(s)=\int \left\{ z_{0}(u(t),v(t))-z_{0}(u(s),v(s))\right\} d\mu
^{1}(u(0))d\mu ^{2}(v(0)).
\end{equation*}

Applying Lemma \ref{t:zbe}, first the zero-balance equation (A1) for $m=0$, $%
n=1$, and then the\ $T$-invariance property (A2)\ of the crossing functions $%
c_{0},c_{1}$ and the fact that $\mu ^{1},\mu ^{2}$ are translationally
invariant, we easily get%
\begin{equation}
Z(t)-Z(s)=-\sum_{r\in \mathbb{Z}}\int d_{0}(u-v+r;s,t)d\mu ^{1}(u(0))d\mu
^{2}(v(0))  \label{r:dissipate}
\end{equation}%
which is $\leq 0$. The case $Z(\mu (t))$ follows by inserting $\mu ^{1}=\mu
^{2}$.
\end{proof}

The relation (\ref{r:dissipate})\ will be very useful in the following, as
it relates the occurrence of singular zeroes with the rate of decrease of $Z$%
. For future reference, because of translation invariance (Lemma \ref{t:zbe}%
, (A2)), for any $0\leq s<t$, and integers $m<n$, (\ref{r:dissipate}) can be
written as (using again shorthand (\ref{r:shorthand}))%
\begin{equation}
Z(t)-Z(s)=-\frac{1}{n-m}\sum_{i=m}^{n-1}\sum_{r\in \mathbb{Z}}\int
d_{i}(u-v+r;s,t)d\mu ^{1}(u(0))d\mu ^{2}(v(0)).  \label{r:dissipate2}
\end{equation}

We will say that two measures $\mu ^{1},\mu ^{2}\in \mathcal{M}(\mathcal{K})$
have a \textit{non-transversal intersection}, if there exist $u^{1},u^{2}$
in their supports so that $u^{1}-u^{2}+r$ has a singular zero for some $r\in 
\mathbb{Z}$. A non-transversal intersection is \textit{proper}, if there is
a singular zero as above with a finite degree.

We also say that any non-increasing function $t\rightarrow f(t)$ is strictly
decreasing at $t_{0}$, if for each $\varepsilon >0\,\ $there exists $\delta
>0$ so that $f(t_{0}+\varepsilon )<f(t_{0}-\varepsilon )-\delta $.

We now show that a proper non-transversal intersection implies that $Z(\mu
^{1},\mu ^{2})$ is strictly decreasing and that this also holds for small
perturbations in the space of measures.

\begin{proposition}
\label{p:strict1}Assume $\mu ^{1}(t_{0}),\mu ^{2}(t_{0})\in \mathcal{M}(%
\mathcal{K)}$ have a proper non-transversal intersection. Then there exist
neighborhoods $U_{1}^{\ast },U_{2}^{\ast }\,\ $of $\mu ^{1}(t_{0}),\mu
^{2}(t_{0})$ in $\mathcal{M}(\mathcal{K})$ so that for all $\nu
^{1}(t_{0})\in U_{1}^{\ast }$, $\nu ^{2}(t_{0})\in U_{2}^{\ast }$, $t\mapsto
Z(\nu ^{1}(t),\nu ^{2}(t))$ is strictly decreasing at $t_{0}$.
\end{proposition}

\begin{proof}
Without loss of generality set $t_{0}=0$, and choose $u^{1},u^{2}\in K$ in
the supports of $\mu ^{1}(0),\mu ^{2}(0)$ respectively, so that $u^{1}-u^{2}$
has a singular zero of a finite degree $k\in \mathbb{N}$ at $%
(m,t_{0}),...,(m+k-1,t_{0})$. By Lemma \ref{l:finite}, we find open
neighborhoods $U^{1},U^{2}$ of $u^{1},u^{2}$ so that (\ref{r:finite}) holds
for any $v^{1}(0),v^{2}(0)\in U^{1}\times U^{2}$. As $u^{1},u^{2}$ are in
the supports of $U^{1},U^{2}$ respectively, $\mu ^{1}(U^{1})\geq 2\delta
_{0} $, $\mu ^{2}(U^{2})\geq 2\delta _{0}$ for some $\delta _{0}>0$.

Let $U_{1}^{\ast },U_{2}^{\ast }$ be the neighborhoods of $\mu
^{1}(t_{0}),\mu ^{2}(t_{0})$ in $\mathcal{M}(\mathcal{K)}$ so that for all $%
\nu ^{1}\in U_{1}^{\ast }$, $\nu ^{2}\in U_{2}^{\ast }$,%
\begin{equation}
\nu ^{1}(U^{1})\geq \delta _{0}\text{, }\nu ^{2}(U^{2})\geq \delta _{0}\text{%
.}  \label{r:measure}
\end{equation}

Now if $Z(t)=Z(\nu ^{1}(t),\nu ^{2}(t))$, combining (\ref{r:dissipate2})
(summing over $m+1,...,m+k$, with $t=\varepsilon $, $s=-\varepsilon $) with (%
\ref{r:finite}) we get%
\begin{eqnarray*}
Z(\varepsilon )-Z(\varepsilon ) &\leq &-\frac{1}{2}\int \mathbf{1}%
_{U^{1}}(v^{1}(0))\mathbf{1}_{U^{2}}(v^{2}(0))d\mu ^{1}(v^{1}(0))d\mu
^{2}(v^{2}(0)) \\
&=&-\frac{1}{2}\nu ^{1}(U^{1})\nu ^{2}(U^{2})\leq -\frac{1}{2}\delta _{0}^{2}%
\text{,}
\end{eqnarray*}%
where $\mathbf{1}_{U}$ denotes the characteristic function of a set $U$.
\end{proof}

We now extend this to functions $Z(\mu (t))$, $\widetilde{Z}(\mu (t))$.

\begin{proposition}
\label{p:strict2}Assume $\mu (t_{0})\in \mathcal{M}(\mathcal{K)}$ has a
proper non-transversal intersection with itself. Then there exists a
neighborhood $U^{\ast }$ of $\mu (t_{0})$ in $\mathcal{M}(\mathcal{K})$ so
that for all $\nu (t_{0})\in U^{\ast }$, $t\mapsto Z(\nu (t))$ is strictly
decreasing at $t_{0}$.
\end{proposition}

\begin{proof}
This follows from Proposition \ref{p:strict1}, by inserting $\mu =\mu
^{1}=\mu ^{2}$ and taking $U^{\ast }=U_{1}^{\ast }\cap U_{2}^{\ast }$.
\end{proof}

\begin{proposition}
\label{p:strict3}Assume $F$ in (\ref{r:FK})\ is DC. Choose $\mu (t_{0})\in 
\mathcal{M}(\mathcal{K)}$ with $u\in \mathcal{K}$ in its support such that $%
du/dt$ has a singular zero of a finite degree. Then $t\mapsto \widetilde{Z}%
(\mu (t))$ is strictly decreasing at $t_{0}$.
\end{proposition}

\begin{proof}
Analogously as Proposition \ref{p:strict2}, applying Lemma \ref{l:der1}
instead of \ref{l:finite}.
\end{proof}

\section{The characteristic flow-map and maps \label{s:characteristic}}

This section is dedicated to the proofs of Theorems \ref{t:main1} and \ref%
{t:main2}. Recall that a probability measure $\mu $ on $\mathcal{K}$ is $%
(\varphi ,T)$-invariant in the DC\ case, if it is invariant for all $\varphi
^{t}$, $t\geq 0$ (time invariance) and for $T$ (space invariance); and in
the AC\ case if invariant for $\varphi ^{1}$ and $T$. We say that a set $%
A\subseteq \mathcal{K}$ is $(\varphi ,T)$-invariant, if it is in the DC\
case invariant for all $\varphi ^{t}$, $t\geq 0$ and $T$;\ and in the AC\
case invariant for $\varphi ^{1}$ and $T$. Keeping in mind this, we will
continue to treat the AC\ and DC\ cases simultaneously.

As ergodic theory for a commuting semiflow (in the AC case non-autonomous)\
and a homeomorphism is not standard, we give a few comments. First, it is
not a-priori clear that $(\varphi ,T)$-invariant measures exist on a
compact, $(\varphi ,T)$-invariant set such as $\mathcal{K}$. We address this
in the next section, and construct a rich set of invariant measures.

A $(\varphi ,T)$-invariant measure $\mu $ is ergodic, if for any Borel
measurable $(\varphi ,T)$-invariant set $A$, $\mu (A)\in \{0,1\}$. An
analogue to the ergodic decomposition theorem holds (\cite{Katok05},
Proposition 4.1.12), so we will in some instances without loss of generality
assume that $\mu \in \mathcal{M(K)}$ is ergodic. Note, for example, that
given any ergodic measure $\mu \in \mathcal{M(X)}$, by ergodicity there is $%
n $ large enough so that $\mu (\mathcal{K}_{n})=1$, thus \mbox{supp}$(\mu
)\subseteq \mathcal{K}_{n}$ as $\mathcal{K}_{n}$ is closed. Therefore
assuming $\mu \in \mathcal{M(K)}$ rather than $\mu \in \mathcal{M(X)}$ is
not a restriction.

Also note that, given a $(\varphi ,T)$-ergodic measure, the Birkhoff
averages with respect to $\varphi ^{t}$ or $T$ do not necessarily $\mu $%
-a.e. converge to the same value. The reason is that a $(\varphi ,T)$%
-ergodic measure is not necessarily either $\varphi $-ergodic or $T$-ergodic.

First we show that $\varphi $ is invertible on the support of any $(\varphi
,T)$-invariant measure.

\begin{lemma}
\label{l:flow}If $\mu $ is $(\varphi ,T)$-invariant supported on $\mathcal{K}%
_{n}$ for some $n\in \mathbb{N}$, then in the DC case $\varphi ^{t}|_{\text{%
supp}(\mu )}$ is a continuous flow, and in the AC case $\varphi ^{1}|_{\text{%
supp}(\mu )}$ is a homeomorphism.
\end{lemma}

\begin{proof}
As \mbox{supp}$(\mu )$ is $(\varphi ,T)$-invariant and $\varphi $ is
continuous, it is easy to see that for any $t\geq 0$ (in the AC case $t$ an
integer), $\varphi ^{t}$(supp$(\mu ))=$supp$(\mu )$. By Lemma \ref%
{l:existence}, $\varphi $ is a local flow, thus injective. We conclude that $%
\varphi ^{t}$ is bijective on supp$(\mu )$. As $\mathcal{X}$ is metrizable
(a quotient space of a metric space, with metric induced by any norm (\ref%
{r:norm})) and $\mathcal{K}_{n}$ is compact, supp$(\mu )$ is compact and $%
\varphi ^{t}:$supp$(\mu )\rightarrow $supp$(\mu )$ is a homeomorphism for
all $t$.
\end{proof}

\begin{proof}[Proof of Theorem \protect\ref{t:main1}.]
Assume $\mu \in \mathcal{M(K)}$ is $(\varphi ,T)$-invariant. First we show
that $t\mapsto Z(\mu (t))$ (where $\mu (0)=\mu $) is constant. In the DC
case, this must be true as $\mu (t)=\mu (0)$ for all $t\geq 0$. In the AC
case, this follows from $\mu (n)=\mu (0)$ for integer $n$, and Lemma \ref%
{p:decreasing1} which showed that $t\mapsto Z(\mu (t))$ is non-decreasing.

It is clear from the definition that the map $\pi :\mathcal{K\rightarrow }%
\mathbb{T}\times \mathbb{R}$ defined with (\ref{r:pi}) is continuous. We
first show that $\pi |_{\text{supp}(\mu )}\,\ $is injective. Indeed if $%
u^{1},u^{2}\in $supp$(\mu )$ are such that $\pi (u^{1})=\pi (u^{2})$, then $%
u^{1}-u^{2}+r$ has a singular zero of a finite or infinite degree for some $%
r\in \mathbb{Z}$. If the degree is finite, by Proposition \ref{p:strict2}, $%
t\mapsto Z(\nu (t))$ is strictly decreasing at $t=0$, where $\nu (0)=\mu $,
which is a contradiction. Now if the degree of the zero is infinite in the
DC\ case, by Lemma \ref{l:infinite} and $\varphi $-invariance of $\mu $, one
can find other $v^{1},v^{2}\in $supp$(\mu )$ so that $v^{1}-v^{2}+r$ has a
singular zero of a finite degree, which reduces it to the previous case.

Consider the case of a singular zero of infinite degree in the AC case. By
Lemma \ref{l:infinite}, we find $v^{1}(t_{0})-v^{2}(t_{1})$ with a singular
zero of a finite degree such that $v^{1}(0),v^{2}(0)\in $supp$(\mu )$. By
the same argument as above, we show that $t\mapsto Z(\mu (t+t_{0}),\mu
(t+t_{1}))$ is non-decreasing with period 1, thus constant, which is a
contradiction with Proposition \ref{p:strict1} and $v^{1}(t_{0})\in $supp$%
(\mu (t_{0}))$, $v^{2}(t_{1})\in $supp$(\mu (t_{1}))$.

Now assume the contrary, i.e. that there are two $(\varphi ,T)$-invariant
measures $\mu ^{1},\mu ^{2}\in \mathcal{M(K)}$, so that for some $u^{1}\in $%
supp$(\mu ^{1})$, $u^{2}\in $supp$(\mu ^{2})$, we have $\pi (u^{1})=\pi
(u^{2})$. But $\mu =\mu ^{1}/2+\mu ^{2}/2$ is also a $(\varphi ,T)$%
-invariant measure whose support contains both $u^{1},u^{2}$, which is
impossible.

Finally, the case $\mu \in \mathcal{M(X)}$ reduces by the ergodic
decomposition theorem to analysis on $\mathcal{M(K}_{n}\mathcal{)}$ for
sufficiently large $n$.
\end{proof}

\begin{remark}
\label{r:one}We actually proved more in the proof above: if $u,v\in \mathcal{%
A}$, then $u-v$ can not\ have a singular zero.
\end{remark}

\begin{proof}[Proof of Theorem \protect\ref{t:main2}.]
By definition, $\pi :\mathcal{K}\rightarrow \widetilde{\mathcal{K}}$ is
continuous and by Theorem \ref{t:main1} it is bijective. As $\mathcal{K}$ is
compact and $\widetilde{\mathcal{K}}$ a metric space, $\pi :\mathcal{K}%
\rightarrow \widetilde{\mathcal{K}}$ is a homeomorphism. The claim now
follows from Lemma \ref{l:flow} and definition (\ref{r:DEF}) of $\widetilde{T%
},\widetilde{\varphi }$.
\end{proof}

We complete the Section with a conjecture and a couple of comments on the
likely structure of the characteristic map $\widetilde{T}$ associated to (%
\ref{r:FK}).

\begin{conjecture}
The map $\widetilde{T}$ can be canonically extended to a homeomorphism $%
\widetilde{T}:\mathbb{T}^{1}\times \mathbb{R}\rightarrow \mathbb{T}%
^{1}\times \mathbb{R}$, diffeomorphic to an area-preserving twist
diffeomorphism $f_{V,F}$.
\end{conjecture}

The numerical simulations \cite{Braun:04} and our construction suggest the
following description of $\widetilde{T}$. Given an AC forcing $F$, we can
define%
\begin{eqnarray*}
\overline{F} &=&\int_{0}^{1}F(t)dt, \\
\sigma (F) &=&\left( \int_{0}^{1}(F(t)-\overline{F})^{2}dt\right) ^{1/2}.
\end{eqnarray*}

We conjecture that for a fixed $V,\sigma (F)$ (including the DC\ case $%
\sigma (F)=0$), there is $F_{c}$ such that for all $\overline{F}\geq F_{c}$,
the map $f_{V,F}$ is integrable. Conversely, for a fixed $V,\overline{F}$,
the limit $\sigma (F)\rightarrow \infty $ corresponds to an "anti-integrable
limit" of the family of maps $f_{V,F}$.

\section{The Aubry-Mather theorem \label{s:amt}}

This section is dedicated to the proof of Theorem \ref{t:main3}. We first
define rotationally ordered configurations in $\mathbb{R}^{Z}$ and the
rotation number and deduce their main properties, then show in general
existence of $(\varphi ,T)$-invariant measures, and from this complete the
proof.

Recall the definitions of translations $T(u)_{i}=u_{i+1}\ $and $Ru=u+1$
commuting with (\ref{r:FK}) and each other. We can define $%
T_{p,q}=T^{p}R^{q} $, or $T_{p,q}(u)_{i}=u_{i+p}+q$ for any $p,q\in \mathbb{Z%
}.\,$We say that a configuration $u\in \mathbb{R}^{Z}$ is rotationally
ordered, if for any $p,q\in \mathbb{Z}$, either $T_{p,q}u\leq u$ or $%
T_{p,q}u\geq u$ (i.e. the set of all $T_{p,q}$ translations is totally
ordered).

The following is an elementary result of Aubry-Mather theory; for a proof
see e.g. Bangert \cite{Bangert88}, Corollary 3.2:

\begin{lemma}
\label{l:roc}If $u\in \mathbb{R}^{Z}$ is a rotationally ordered
configuration, then there exists a unique $\rho \in \mathbb{R}$ so that for
all $i,j\in \mathbb{Z}$,%
\begin{equation}
|u_{j}-u_{i}-\rho (j-i)|\leq 1.  \label{r:width}
\end{equation}
\end{lemma}

The number$\rho $ can be characterized as $\rho =\sup \{p/q:T_{p,q}u\leq
u\}=\inf \{p/q:T_{p,q}u\geq u\}$.

If $u\in \mathbb{R}^{Z}$ is such that the limit%
\begin{equation*}
\rho (u)=\lim_{|j-i|\rightarrow \infty }\frac{u_{j}-u_{i}}{j-i}
\end{equation*}%
exists, we call it the rotation number (or mean spacing) of a configuration.
For rotationally ordered configurations, $\rho \,$in (\ref{r:width}) is its
rotation number.

Now we show existence of $(\varphi ,T)$-invariant measures.

\begin{lemma}
\label{r:existence}Assume $\mathcal{R}\subset \mathcal{K}$ is a closed,
non-empty $(\varphi ,T)$-invariant set. Then there exists a $(\varphi ,T)$%
-invariant probability measure supported on $\mathcal{R}$.
\end{lemma}

\begin{proof}
Choose any Borel probability measure $\nu $ on $\mathcal{R}$ (say Dirac
supported on any element).\ We in the DC\ case define 
\begin{equation*}
\nu ^{n}=\frac{1}{n^{2}}\int_{0}^{n}\left( \sum_{k=1}^{n}\varphi _{\ast
}^{t}T_{\ast }^{k}\nu \right) dt,
\end{equation*}%
and in the AC\ case 
\begin{equation*}
\nu ^{n}=\frac{1}{n^{2}}\sum_{j=1}^{n}\sum_{k=1}^{n}\varphi _{\ast
}^{j}T_{\ast }^{k}\nu dt.
\end{equation*}%
Here $\varphi _{\ast }^{t},T_{\ast }^{k}$ is pulling of measures with
respect to $\varphi ^{t},T^{k}$. As $\mathcal{R}$ is compact and $\nu ^{n}$
are well defined probability measures supported on $\mathcal{R}$, $\nu ^{n}$
has a weakly convergent subsequence $\nu ^{n_{k}}\rightarrow \mu $ and $\mu $
is supported on $\mathcal{R}$. It is easy to check that $\mu $ is $(\varphi
,T)$-invariant (e.g. by checking this for continuos functions $f:\mathcal{%
R\rightarrow }\mathbb{R}$ and then applying the Riesz representation
Theorem).
\end{proof}

\begin{proof}[Proof of Theorem \protect\ref{t:main3}.]
Denote by $\mathcal{R}_{\rho }$ the set of all rotationally ordered
configurations with the rotation number $\rho $. It is non-empty, as $u\,\ $%
defined with $u_{i}=\rho i$ is in $\mathcal{R}_{\rho }$. From (\ref{r:width}%
) one immediately sees that $\mathcal{R}_{\rho }\subset \mathcal{K}_{n}$
where $n$ is any integer $\geq |\rho |+1$. From the definition it follows
that $\mathcal{R}_{\rho }$ is closed, thus it is compact. As $T_{p,q}$
commutes with $\varphi $ and $T$, by Lemma \ref{p:monotonicity} $\varphi $
is order preserving and $T$ is trivially order preserving, we see that $%
\mathcal{R}_{\rho }$ is $(\varphi ,T)$-invariant. It is now enough to apply
Lemma \ref{r:existence} to $\mathcal{R}_{\rho }$.
\end{proof}

\section{Description of invariant measures in the\ DC\ case \label{s:DCcase}}

Here we prove Theorem \ref{t:main4}, and so describe the set of $(\varphi ,T)
$-invariant measures on $\mathcal{X}$ and $\mathcal{K}$, denoted by $%
\mathcal{E(X)}$ and $\mathcal{E(K)}$, in the DC case. The standing
assumption in this section is that $F$ in (\ref{r:FK})\ is DC. We first
define the average speed of a measure and a configuration. We then show by
applying the Poincar\'{e}-Bendixson theorem that any $\mu \in \mathcal{E}(K)$
is entirely supported on equilibria and periodic orbits. If the average
speed is $\not=0$, then $\mu $ is supported on a single periodic orbits.
Finally, if the average speed is zero, we show that $\mu $ is supported on
equilibria.

Let $\mu \in \mathcal{E(X)}$ and $u\in \mathcal{X}$. We can define their
average speeds $v(\mu )$, $v(u)$ as%
\begin{eqnarray}
v(\mu ) &=&\int (u_{0}(1)-u_{0}(0))d\mu \text{,}  \label{r:speed1} \\
v(u) &=&\lim_{t\rightarrow \infty }\frac{u_{0}(t)-u_{0}(0)}{t},
\label{r:speed2}
\end{eqnarray}%
where $v(u)$ is defined when (\ref{r:speed2}) is convergent.

\begin{lemma}
\label{l:speed}If $\mu $ is $(\varphi ,T)$-ergodic on $\mathcal{X}$, then
for $\mu $-a.e. $u$, $v(u)=v(\mu )$.
\end{lemma}

\begin{proof}
As $\mu $ is $\varphi $-invariant (but not necessarily $\varphi $-ergodic!),
there is a set of full measure such that $v(u)$ is well defined. As by
definition and Lemma \ref{l:existence}\ the spacing $|u_{i+1}(t)-u_{i}(t)|$
is bounded uniformly in $i,t$, the definition (\ref{r:speed2})\ is
independent of $i$ in $u_{i}$, thus $v(u)=v(Tu)$. Thus $v(u)$ is defined on
a set of full measure, $\varphi $ and $T$-invariant measurable function, so
it must be constant $\mu $-a.e. and equal to $v(\mu )$.
\end{proof}

We will use a version of the well-known Poincar\'{e}-Bendixson theorem on
asymptotics of 2-dimensional flows. For a general discussion of the Poincar%
\'{e}-Bendixson theorem see e.g. \cite{Katok05}, Section 14.1;\ for the
proof of the exact formulation below see \cite{Fiedler89}, the proof of
Proposition 1 (the only difference being that they consider compact subsets
of the plane, which makes no difference, as the argument works for any
compact sets which can be embedded in a two dimensional sphere).

\begin{lemma}
\label{l:uc1}Assume $\widetilde{K}$ is a compact subset of the cylinder $%
\mathbb{T}^{1}\times \mathbb{R}$, and $\tilde{\varphi}$ a complete
continuous flow on $\widetilde{K}$ generated by a $C^{0}$ vector field on $%
\widetilde{K}$. Then the only recurrent points of $\tilde{\varphi}$ are
equilibria and periodic orbits.
\end{lemma}

We need the next lemma to be able to claim that if $\pi (u)$ is an
equilibrium, then $u$ must be an equilibrium.

\begin{lemma}
\label{l:uc2}If $u\in \mathcal{A}$, then $w=du/dt$ can not have a singular
zero.
\end{lemma}

\begin{proof}
Find a measure $\mu \in \mathcal{E(K)\,}$for some $\mathcal{K=K}_{n}\,\ $so
that $u$ is in the support of $\mu $. As $\mu =\mu (0)$ is by definition $%
(\varphi ,T)$-invariant, $t\mapsto \tilde{Z}(\mu (t))$ is a constant
function. Now if $w=du/dt$ has a singular zero of a finite degree, from
Proposition \ref{p:strict3} we obtain that $\tilde{Z}(\mu (t))$ is strictly
decreasing at $t=0$, which is a contradiction. If $w$ has a singular zero of
infinite degree, by Lemma \ref{l:der2} we can find $v^{1},v^{2}$ in the
support of $\mu $ so that $v^{1}-v^{2}$ has a singular zero of finite
degree.\ This implies that by Proposition \ref{p:strict1}, $t\mapsto Z(\mu
(t))$ is strictly decreasing at $t=0$, which is again a contradiction with $%
\mu $ being $(\varphi ,T)$-invariant.
\end{proof}

We now show that all of this combined with our results on the projection $%
\pi $ gives a fairly good description of supports of $\mu \in \mathcal{E(X)}$%
.

\begin{lemma}
\label{l:uc3}Assume $\mu $ is $(\varphi ,T)$-ergodic. Then $\mu $ is
supported on $\varphi $-periodic orbits with the same period $t_{0}>0$, or
on equilibria.
\end{lemma}

\begin{proof}
As $\mu $ is $(\varphi ,T)$-ergodic, there exists $\mathcal{K}=\mathcal{K}%
_{n}$ for $n$ large enough so that $\mu \in \mathcal{E(K)}$. Take $%
\widetilde{K}$, $\tilde{\varphi}$ as in (\ref{r:DEF}). Now we deduce from
Theorems \ref{t:main1} and \ref{t:main2} that most of the assumptions of
Lemma \ref{l:uc1} are satisfied;\ we only need to show that $\tilde{\varphi}$
is generated by a $C^{0}$ vector field. One can easily check that for $%
(x,p)\in \mathbb{T}^{1}\times \mathbb{R}$, the vector field is given by%
\begin{equation}
(x,p)\mapsto \left[ 
\begin{array}{l}
d(\pi ^{-1}(x,p))_{0}/dt \\ 
d(\pi ^{-1}(x,p))_{1}/dt-d(\pi ^{-1}(x,p))_{0}/dt%
\end{array}%
\right] ,  \label{r:vf}
\end{equation}%
where $du_{0}/dt$, $du_{1}/dt$ for $u=\pi ^{-1}(x,p)$ are defined with the
right-hand side of (\ref{r:FK}). From Lemma \ref{l:uc1} we deduce that all $%
\tilde{\varphi}$-recurrent orbits on $\widetilde{K}$ are periodic orbits and
equilibria. Consider the pulled $\tilde{\varphi}$-invariant measure $\pi
_{\ast }\mu $ on $\widetilde{K}$. As $\widetilde{K}$ is metrizable and
compact, the set of $\tilde{\varphi}$-recurrent points has full measure with
respect to $\pi _{\ast }\mu $ (see e.g. \cite{Katok05}, the proof of
Proposition 4.1.18). Now note that by Theorem \ref{t:main1}, $\pi ^{-1}$
maps periodic orbits into periodic orbits, and by Lemma \ref{l:uc2}, $\pi
^{-1}$ maps equilibria into equilibria (otherwise the right-hand side of (%
\ref{r:vf}) would be zero, i.e. for $u=\pi ^{-1}(x,p)$, $du/dt$ would have a
singular zero, which is impossible by Lemma \ref{l:uc2}). If we denote by $%
\mathcal{P}_{t}$ the orbits in $\mathcal{K}$ with period $t$, $t\geq 0$
(where $\mathcal{P}_{0}$ are equilibria), we now have%
\begin{equation}
\mu \left( \bigcup_{t\geq 0}\mathcal{P}_{t}\right) =1\text{.}  \label{r:cont}
\end{equation}

Now for each $t\geq 0$, $\mathcal{P}_{t}$ is $(\varphi ,T)$-invariant, thus
by ergodicity $\mu (\cup _{t\leq t_{0}}\mathcal{P}_{t})\in \{0,1\}$. As $%
t_{0}\longmapsto \mu (\cup _{t\leq t_{0}}\mathcal{P}_{t})$ is
non-decreasing, there exists a unique $t_{0}$ so that $\mu (\mathcal{P}%
_{t_{0}})=1$. As $\mathcal{P}_{t_{0}}$ is closed, it contains the support of 
$\mu $.
\end{proof}

We finally need to show that, if $\mu \,$is supported on periodic orbits and
has non-zero speed, it is supported on a single one.

We say that the orbit of $u$ is totally ordered, if the set $\Upsilon
(u)=\{\varphi ^{t}T^{k}u,t\in \mathbb{R},k\in \mathbb{Z\}}$ is totally
ordered. That means that for any $v,w\in \Upsilon (u)$, we have either $%
v\geq w$ or $v\leq w$.

\begin{lemma}
\label{l:single} Assume $\mu $ is $(\varphi ,T)$-ergodic, $v(\mu )\not=0$,
and supported on periodic orbits with period $t_{0}>0$. Then it is supported
on a single totally ordered periodic orbit such that the $\pi \circ T^{j}$%
-image of supp$(\mu )$ is independent of $j\in \mathbb{Z}$.
\end{lemma}

\begin{proof}
We prove it in three steps. First we show that $u(t_{0})=u(0)\pm 1$, then we
show that $u(0)$ and $u(t)$ in the support of $\mu $ can not intersect, and
finally we show uniqueness of the periodic orbit.

The average non-zero speed is by Lemma \ref{l:speed} the same for all the
periodic orbits in supp$(\mu )$. By Theorem \ref{t:main1}, $\pi $ projects
each of them to a homotopically non-trivial invariant circle on $\mathbb{T}%
^{1}\times \mathbb{R}$ (otherwise the speed would be zero), thus $%
u(t_{0})=u(0)\pm 1$, with the sign the same as the sign of $v(\mu )$. Assume
without loss of generality $u(t_{0})=u(0)+1$.

Now assume $u(0)$ and $u(t_{1})$ intersect, $0<t_{1}<t_{0}$ (the only
analogous alternative is $-t_{0}<t_{1}<0$). Consider $w(t)=u(t)-u(0)$.
Recall that by Theorem \ref{t:main1} and Remark \ref{r:one}, all
intersections in the support of $\mu $ are transversal, so $w(t)$ has a
regular zero, say at $i$. Then $w(t)$ intersects zero at%
\begin{equation}
x(t)=\frac{(i+1)u_{i}-i(u_{i+1})}{u_{i}-u_{i+1}}.  \label{r:cont1}
\end{equation}

As $w(t)$ has no singular zeroes, it is easy to see we can continue the
function $t\mapsto x(t)$ locally with (\ref{r:cont1}) if $x(t)$ not an
integer, and if $x(t)=i$ an integer, with%
\begin{equation*}
x(t)=\left\{ 
\begin{array}{ll}
\frac{(i+1)u_{i}-i(u_{i+1})}{u_{i}-u_{i+1}} & u_{i}(t)\geq 0, \\ 
\frac{iu_{i-1}-(i-1)u_{i}}{u_{i-1}-u_{i}} & u_{i}(t)<0.%
\end{array}%
\right.
\end{equation*}

Thus we can find a continuous function $t\mapsto x(t)$ so that $w(t)$ has a
regular zero at $\left\lfloor x(t)\right\rfloor $. But by assumptions $%
w(t_{0})\equiv 1$, so the only possibility is that for some $t_{2}$, $%
t_{1}<t_{2}<t_{0}$, $\lim_{t\rightarrow t_{2}^{-}}x(t)=+\infty $ or $%
\lim_{t\rightarrow t_{2}^{-}}x(t)=-\infty $. Assume without loss of
generality $\lim_{t\rightarrow t_{2}^{-}}x(t)=+\infty $. As by the proof of
Lemma \ref{l:existence}, $t\mapsto w(t)$ is continuous in $l_{\infty }$, we
easily see that $\lim_{i\rightarrow \infty }w_{i}(t_{2})=0$, thus 
\begin{equation}
\lim_{i\rightarrow \infty }|u_{i}(t_{2})-u_{i}(0)|=0.  \label{r:asymmm}
\end{equation}%
By compactness, we can find a sequence $T^{n_{k}}$, $n_{k}\rightarrow \infty 
$ such that $T^{n_{k}}u(0)$ converges to some $v$. As $T$ commutes with $%
\varphi $, from (\ref{r:asymmm}) we see that $v(t_{2})=v(0)$, and $v$ is in
the support of $\mu $. But $t_{2}<t_{0}$, which is in contradiction with the
assumption.

Now we see that each periodic orbit in the $\pi $-image of support of $\mu $
intersects each $x$ coordinate exactly once at a point $p$, where $(x,p)\in 
\mathbb{T}^{1}\times \mathbb{R}$. We can thus identify the quotient set supp$%
(\mu )/\varphi $ with a compact set $\bar{B}$ of intersections at $x=0$, and
consider $\bar{B}$ to be a subset of $\mathbb{R}$ parametrized with $p$. By
definitions, $T$ induces a homeomorphism $\bar{T}$ on $\bar{B}$, and the
induced flow $\bar{\varphi}$ is constant on $\bar{B}$. The induced measure $%
\bar{\mu}$ is $(\bar{T},\bar{\varphi})$-ergodic on $\bar{B}$, and as $\bar{%
\varphi}$ is constant, it must be $\bar{T}$ ergodic. Finally, note that $%
\bar{T}$ is order preserving on $\bar{B}$, that means that $p>q$ implies $%
\bar{T}(p)>\bar{T}(q)$. We see that, as otherwise we find $u,v\,\ $in
equivalence classes of $p,q$ which violate order preserving, such that $%
u_{0}(t)=v_{0}(t^{\prime })$, and then $u(t)-v(t^{\prime })$ has a singular
zero at $0$.\ As by assumptions, $\bar{B}$ is the support of $\bar{\mu}$, $%
\bar{T}$ a homeomorphism and $\bar{\mu}$ is $\bar{T}$-ergodic, there must be 
$p_{0}\in \bar{B}$ such that its $\omega $-limit set with respect to $\bar{T}
$ is the entire $\bar{B}$. But $\bar{T}^{n}(p_{0})$ is monotone and bounded,
thus converges to a single point. We conclude that $\bar{B}$ must be a
single point, so support of $\mu $ consists of a single periodic orbit.

It must be a totally ordered set, as we have shown that $u(t),u(t^{\prime })$
do not intersect unless equal.
\end{proof}

\begin{proof}[Proof of Theorem \protect\ref{t:main4}.]
Take any $(\varphi ,T)$-ergodic measure $\mu $ on $\mathcal{X}$. If $v(\mu
)=0$, consider the function $\int V(u_{0},u_{1})d\mu $. The derivatives of $%
V $ and $u_{i}$ are by compactness uniformly bounded on the semiorbit of $%
u(t)$. Thus we can exchange the integral and derivation, use $T$-invariance
of $\mu $, (\ref{r:FK}) and finally $\varphi $-invariance of $\mu $ and
obtain%
\begin{eqnarray}
\frac{d}{dt}\int V(u_{0},u_{1})d\mu &=&\int \left(
V_{1}(u_{0},u_{1})u_{0}^{\prime }(t)+V_{2}(u_{0},u_{1})u_{1}^{\prime
}(t)\right) d\mu  \notag \\
&=&\int \left( V_{1}(u_{0},u_{1})u_{0}^{\prime
}(t)+V_{2}(u_{-1},u_{0})u_{0}^{\prime }(t)\right) d\mu  \notag \\
&=&-\int \left( u_{0}^{\prime }(t)\right) ^{2}d\mu +F\int u_{0}^{^{\prime
}}(t)d\mu  \notag \\
&=&-\int \left( u_{0}^{\prime }(t)\right) ^{2}d\mu +F\int \left(
\int_{0}^{1}u_{0}^{^{\prime }}(t)dt\right) d\mu  \notag \\
&=&-\int \left( u_{0}^{\prime }(t)\right) ^{2}d\mu +Fv(\mu )\text{.}
\label{r:vzero}
\end{eqnarray}

However $V$ is continuous on $\mathcal{X}$ and $\mu $ is $\varphi $%
-invariant, so the left-hand side of (\ref{r:vzero}) must be $0$. We deduce
that for $\mu $-a.e. $u$, $u_{0}^{\prime }(t)=0$, so by continuity argument
and $T$-invariance we see that the support of $\mu $ consists of equilibria.

If $v(\mu )\not=0$, the claim follows from lemmas \ref{l:uc3} and \ref%
{l:single}.
\end{proof}

\begin{remark}
The relation 
\begin{equation*}
F\cdot v(\mu )=\int \left( u_{0}^{\prime }(t)\right) ^{2}d\mu
\end{equation*}%
obtained above, thus valid for DC dynamics, has standard physical
interpretations, such as force$\times $speed=energy dissipation, or voltage$%
\times $current=resistance.
\end{remark}

\section{Uniformly sliding states in the DC\ case \label{s:sliding}}

In this section we further discuss the DC\ case when the average speed is
not zero. We show that the periodic orbit is then a uniformly sliding state,
and discuss the properties of its "modulation function". The results of this
section have already been obtained by Baesens, MacKay and Qin (\cite%
{Baesens98}, \cite{Baesens05}, \cite{Qin10}, \cite{Qin11}), by combination
of a limiting procedure (typically going from rational to irrational
rotation numbers) and a fixed point argument.

We show that these results follow directly from our ergodic-theoretical and
topological considerations, hoping for extension of these results to AC\
driving and systems similar to (\ref{r:FK}).

We say that a solution of (\ref{r:FK})\ is a uniformly sliding solution with
the rotation number $\rho $, speed $v$ and time shift $\alpha $, if there
exists a modulation function $m:\mathbb{R}\rightarrow \mathbb{R}$, $%
m(x+1)=m(x)$, such that for each $j\in \mathbb{Z}$, $t\in \mathbb{R}$,%
\begin{equation*}
u_{j}(t)=j\rho +vt+\alpha +m(j\rho +vt+\alpha ).
\end{equation*}

\begin{corollary}
\label{c:c}Assume that $F$ is DC, $\mu $ is $(\varphi ,T)$-ergodic and $%
v(\mu )\not=0$. Then $\mu ~$is supported on a single uniformly sliding
solution.
\end{corollary}

\begin{proof}
Assume without loss of generality $v(\mu )>0$. As by Lemma \ref{l:single}
and Theorems \ref{t:main1}, \ref{t:main2}, $\pi $ projects $u_{j}(t)$ to a
single periodic orbit with period $t_{0}$ independently of $j$ and with the
same projected flow $\tilde{\varphi}$, there exists a function $\widetilde{m}%
:\mathbb{R}\rightarrow \mathbb{R}$ and $\alpha _{j}\in \mathbb{R}$ such that 
$\widetilde{m}(t+t_{0})=\widetilde{m}(t)+1$ and for all $j\in Z$, $u_{j}(t)=%
\widetilde{m}(t+\alpha _{j})$. It is easy to see that $v=v(\mu )=1/t_{0}$.
We set $m(t):=\widetilde{m}(t\cdot t_{0})-t$, and thus $m(t+1)=m(t)$ and%
\begin{equation}
u_{j}(t)=m(vt+v\alpha _{j})+vt+v\alpha _{j}\text{.}  \label{r:modul}
\end{equation}

As the set $\{u(t),t\in \mathbb{R}\}$ is by Theorem \ref{t:main4} totally
ordered, $u(t)$ is rotationally ordered, so by\ Lemma \ref{l:roc}, (\ref%
{r:width}) must hold with some $\rho =\rho (u(t))$. As $t\longmapsto
u(t)-u(0)$ is continuous in $l_{\infty }(\mathbb{Z})$, all $u(t)$ must have
the same rotation number equal to the rotation number of the measure defined
as%
\begin{eqnarray*}
\rho (\mu ) &:&=\int (u_{1}-u_{0})d\mu \\
&=&v\int \left( \int_{0}^{t_{0}}(u_{i+1}(t)-u_{i}(t))dt\right) d\mu ,
\end{eqnarray*}%
where in the second row we used $T$ and $\varphi $-invariance of $\mu $. By
Lemma \ref{l:single}, $\int_{0}^{t_{0}}(u_{i+1}(t)-u_{i}(t))$ is independent
of $i$ as it is the second coordinate of the projection $\pi $ integrated
over one period, and by periodicity independent of $u\in \sup $p$(\mu )$,
thus for all $i\in \mathbb{Z}$, $t\in \mathbb{R}$,%
\begin{equation}
\int_{0}^{t_{0}}(u_{i+1}(t)-u_{i}(t))dt=\frac{\rho }{v}\text{.}
\label{r:modul2}
\end{equation}

Combining (\ref{r:modul}) and (\ref{r:modul2})\ we get $v\alpha _{j}=j\rho
+\alpha $ for some $\alpha \in \mathbb{R}$, which completes the proof.
\end{proof}

From Theorems \ref{t:main3}, \ref{t:main4} now follow the results from \cite%
{Baesens98}, \cite{Qin10} that if there is no equilibrium with a given
rotation number $\rho \in \mathbb{R}$, there must be an uniformly sliding
solution with the same rotation number. Furthermore, one can prove
continuous interdependence of $v,\rho ,F$ and other parameters of (\ref{r:FK}%
) by considering continuity of $(\varphi ,T)$-invariant measures and
Theorems \ref{t:main3}, \ref{t:main4}.

\section{The spatio-temporal attractor \label{s:attractor}}

In this section we prove Theorem \ref{t:main5}. We need to show that, given
any $T$-invariant probability measure $\mu $ on $\mathcal{X}$ and any open
set neighbourhood $U$ of $\mathcal{X}$, (\ref{d:attractor}) holds. As $%
\mathcal{K}_{n}$ are closed and $T$-invariant, by the ergodic decomposition
theorem it is sufficient to prove that for $\mu \in \mathcal{M}(\mathcal{K}%
_{n})$ (denoted again with $\mathcal{M}(\mathcal{K})$).

Recall the definition of $P_{\mu ,S}$ in Section 2. It is easy to see that
the left-hand side of (\ref{d:attractor}) can be written as%
\begin{equation}
P_{\mu ,S}(\{(u,t),\text{ }\varphi ^{t}(u)\in A\})=\frac{1}{S}%
\int_{0}^{T}(\varphi _{\ast }^{t}\mu (A))dt.  \label{r:mama}
\end{equation}%
Denote the probability measure on the right-hand side of (\ref{r:mama}) with 
$\nu ^{S}$. It is easy to see that $\nu ^{S}\in \mathcal{M}(\mathcal{K})$.
We want to show that, given any open neighbourhood $U$ of $\mathcal{A}$,
then $\lim_{S\rightarrow \infty }\nu ^{S}(U)=1$. Assume the contrary, and
find a sequence $S_{k}\rightarrow \infty $ and $\delta >0$ so that $\nu
^{S_{k}}(U)<1-\delta $ for all $k$. As $\mathcal{M}(\mathcal{K})$ is
compact, $\nu ^{S_{k}}$ has a subsequence for simplicity denoted also with $%
\nu ^{S_{k}}$, which converges to some $\mu \in \mathcal{M}(\mathcal{K})$.
It is easy to deduce from (\ref{r:mama}) that $\mu $ is also $\varphi $%
-invariant. Then $\mu $ must be supported on $\mathcal{A}$, thus $\mu (%
\mathcal{A)=}1$, $\mu (U\mathcal{)}=1$ and as $U$ open, $\lim_{k\rightarrow
\infty }$ $\nu ^{S_{k}}(U)=1$ which is a contradiction.

\section{Appendix 1:\ on the zero-set of a system of infinitely many
cooperative equations}

\subsection{The proof of Lemma \protect\ref{t:zeroset}.}

Recall the setting and definitions from Section \ref{s:zeroset}, and assume
in the following all the assumptions of Lemma \ref{t:zeroset}.

\begin{lemma}
\label{l:one}Assume $w_{1}(0),...,w_{k}(0)=0$ is a singular zero of degree $%
k $. Then there exist real numbers $d_{1},...,d_{k}$ such that for all $%
j=1,...,k$,%
\begin{equation}
w_{j}(t)=d_{j}t^{j^{\ast }}+o(t^{j^{\ast }})\text{,}  \label{r:order}
\end{equation}%
where $j^{\ast }=\min \{j,k+1-j\}$.
\end{lemma}

\begin{proof}
We prove inductively in $l$, $1\leq l\leq (k+1)/2$, that for all $j$, $l\leq
j\leq k+1-l$,%
\begin{equation}
w_{j}(t)=d_{j,l}\cdot t^{l}+o(t^{l})\text{.}  \label{r:der2}
\end{equation}

For $l=1$ and all $1\leq j\leq k$, as $w_{j}$ is $C^{1}$ and $w_{j}(0)=0$,
by the Taylor formula, 
\begin{equation*}
w_{j}(t)=\frac{dw_{j}(0)}{dt}t+o(t)\text{,}
\end{equation*}%
thus 
\begin{equation}
d_{j,1}=dw_{j}(0)/dt.  \label{r:basis}
\end{equation}

We now assume (\ref{r:der2}) holds for some $l-1$. Inserting (\ref{r:der2})
in the right-hand side of (\ref{r:system}) and using (\ref{r:conditions})
(continuity of $a_{j},b_{j},c_{j}$) we get%
\begin{eqnarray}
\frac{d}{dt}w_{j}(t)
&=&(a_{j}(t)d_{j-1,l-1}+b_{j}(t)d_{j+1,l-1}+c_{j}(t)d_{j,l-1})\cdot
t^{l-1}+o(t^{l-1})  \notag \\
&=&(a_{j}(0)d_{j-1,l-1}+b_{j}(0)d_{j+1,l-1}+c_{j}(0)d_{j,l-1})t^{l-1}+o(1)t^{l-1}+o(t^{l-1})
\notag \\
&=&(a_{j}(0)d_{j-1,l-1}+b_{j}(0)d_{j+1,l-1}+c_{j}(0)d_{j,l-1})t^{l-1}+o(t^{l-1})%
\text{.}  \label{r:recurrence}
\end{eqnarray}

Integrating it and using $w_{j}(0)=0$ we obtain (\ref{r:der2}) for $l$ and
complete the inductive step with%
\begin{equation}
d_{j,l}=\frac{1}{l}%
(a_{j}(0)d_{j-1,l-1}+b_{j}(0)d_{j+1,l-1}+c_{j}(0)d_{j,l-1})\text{.}
\label{r:iteration}
\end{equation}
\end{proof}

We now need to show that $d_{j}\not=0$ and determine their sign. Denote by
sgn$(x)\in \{-1,0,1\}$ the sign of $x\in \mathbb{R}$.

\begin{lemma}
\label{l:two}Assume all as in Lemma \ref{l:one}. Then for all $j=1,...,k$,
except in the case $j=k/2$, $k$ even, the zero of Type II, we have%
\begin{equation}
\text{sgn}(d_{j})=\text{sgn}(w_{b}(0))\text{,}  \label{r:sign}
\end{equation}%
where $b\in \{0,k+1\}$ such that $|j-b|\leq |j-(k+1-b)|$.
\end{lemma}

\begin{proof}
From (\ref{r:basis}) and (\ref{r:system})\ we get 
\begin{eqnarray*}
d_{1,1} &=&dw_{1}(0)/dt=a_{1}(0)w_{0}(0), \\
d_{k,1} &=&dw_{k}(0)/dt=b_{k}(0)w_{k+1}(0), \\
d_{j,1} &=&0\text{ for }j=2,...,k-1\text{.}
\end{eqnarray*}

Iteratively from\ (\ref{r:iteration}) we easily get using again $j^{\ast
}=\min \{j,k+1-j\}$%
\begin{equation*}
\begin{array}{ll}
d_{j}=w_{0}(0)\prod_{i=1}^{j}a_{i}(0)/j^{\ast }! & \text{for }j<k/2\text{,}
\\ 
d_{j}=w_{k+1}(0)\prod_{i=j}^{k}b_{i}(0)/j^{\ast }! & \text{for }j>k/2, \\ 
d_{j}=w_{0}(0)\prod_{i=1}^{j}a_{i}(0)/j^{\ast
}!+w_{k+1}(0)\prod_{i=j}^{k}b_{i}(0)/j^{\ast }! & \text{for }j=k/2\text{,}%
\end{array}%
\text{ }
\end{equation*}%
which completes the proof.
\end{proof}

Relations (\ref{r:order}) and (\ref{r:sign}) now complete the description of
the behavior of a solution around a singular zero of degree $k$ for small $t$%
. We summarize it, again assuming without loss of generality (i.e. up to
reparametrization) that $(1,0),...,(k,0)\in S$ is a zero of degree $k$ at $%
t=0$. In that case, lemmas \ref{l:one} and \ref{l:two} imply that, assuming
all as in Lemma \ref{l:one}, that there exists $\delta _{0}>0$ such that for
all $0<\delta \leq \delta _{0}$, we have (see Figure 1):%
\begin{equation*}
\begin{array}{ccccc}
\text{Type of zero} & k & z_{0,k+1}(w(-\delta )) & z_{0,k+1}(w(0)) & 
z_{0,k+1}(w(\delta )) \\ 
\text{I} & \text{even} & k+1 & k & 1 \\ 
\text{I} & \text{odd} & k & k & 1 \\ 
\text{II} & \text{even} & k & k & 0 \\ 
\text{II} & \text{odd} & k+1 & k & 0%
\end{array}%
\end{equation*}

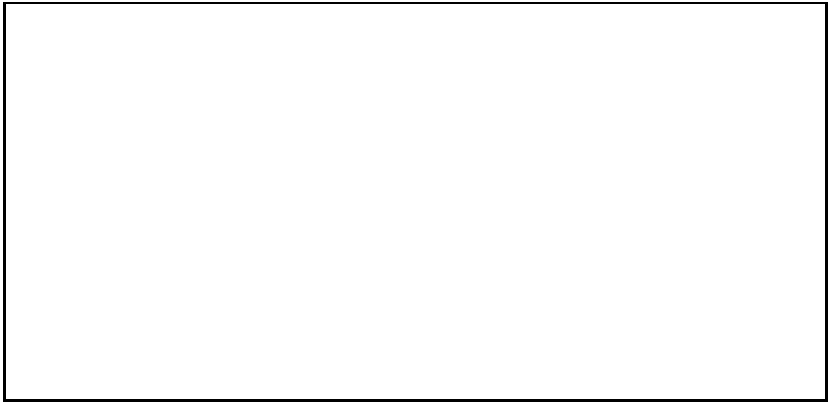
\FRAME{dtbpFUX}{4.2514in}{2.0686in}{0pt}{\Qcb{\textit{Figure 1:\ Types of
singular zeroes of finite degree}}}{\Qlb{f:two}}{picture5.gif}{\special%
{language "Scientific Word";type "GRAPHIC";maintain-aspect-ratio
TRUE;display "USEDEF";valid_file "F";width 4.2514in;height 2.0686in;depth
0pt;original-width 5.2789in;original-height 2.5516in;cropleft "0";croptop
"1";cropright "1";cropbottom "0";filename 'Picture5.gif';file-properties
"XNPEU";}}

We conclude that the number of zeroes is strictly decreasing at singular
zeroes. By continuity, the number of zeroes does not change at regular
zeroes, which completes the proof of Lemma \ref{t:zeroset}.

\subsection{Zeroes of degree $\infty $.}

\textbf{A}s noted earlier, zeroes of degree $\infty $ require special
attention.

\begin{lemma}
\label{l:infitiy}Assume all as in Lemma \ref{t:zeroset}. Suppose $w(0)$ has
a zero of degree $\infty $, $w_{0}(0)\not=0$ and $w_{i}(0)=0$ for all $i\geq
1$. Then there exist $d_{j}$ having the same sign as $w_{0}(0)\,\ $such that
for all $j\geq 1$, 
\begin{equation}
w_{j}(t)=d_{j}t^{j}+o(t^{j})\text{.}  \label{r:auto}
\end{equation}
\end{lemma}

\begin{proof}
An analogous analysis as in the proofs of Lemmas \ref{l:one} and \ref{l:two}
easily yields (\ref{r:auto})\ with%
\begin{equation*}
d_{j}=\frac{w_{0}(0)}{j!}\prod_{i=1}^{j}a_{i}(0).
\end{equation*}%
\bigskip
\end{proof}

\subsection{Proof of Lemma \protect\ref{t:zbe}.}

We construct explicitly the functions $c_{i}(s,t),d_{i}(s,t)$. We first
define sets of zeroes $S^{+}$, $S^{-}$ which are subsets of the set of
singular zeroes $S\subset \mathbb{Z}\times \lbrack T_{0},T_{1}]$ which
"disappear" either at $t=t_{0}\,$or at $t=t_{0}^{+}$. Assume without loss of
generality that $(1,t_{0}),...,(k,t_{0})$ is a singular zero of degree $k$.
Then we define:

(i) Zero of Type I, $k$ even, zero at $i=1,...,2l=k$: 
\begin{eqnarray*}
S^{-} &\ni &(1,t_{0}), \\
S^{+} &\ni &(1,t_{0}),(2,t_{0}),..,(l-1,t_{0}),(l+1,t_{0}),...,(2l,t_{0}).
\end{eqnarray*}

(ii)\ Zero of Type I, $k$ odd, zero at $i=1,...,2l+1=k$: 
\begin{equation*}
S^{+}\ni (1,t_{0}),(2,t_{0}),..,(l,t_{0}),(l+2,t_{0}),...,(2l+1,t_{0}).
\end{equation*}

(iii)\ Zero of Type II, $k$ even, zero at $i=1,...,2l=k$:\ 
\begin{equation*}
S^{+}\ni (1,t_{0}),...,(2l,t_{0}).
\end{equation*}

(iv)\ Zero of Type II, $k$ odd, zero at $i=1,...,2l+1=k$:%
\begin{eqnarray*}
S^{-} &\ni &(l+1,t_{0}), \\
S^{+} &\ni &(1,t_{0}),...,(2l+1,t_{0}).
\end{eqnarray*}

(v) Zero of degree $\infty $ at $i\geq 1$:%
\begin{equation*}
S^{+}\ni (1,t_{0}),(2,t_{0}),...
\end{equation*}

Also $S^{-},S^{+}$ contain no other elements apart from those defined above.
We now define%
\begin{equation*}
d_{i}(s,t)=|S^{+}\cap \{i\}\times \lbrack s,t)|+|S^{-}\cap \{i\}\times
(s,t]|,
\end{equation*}%
where $|A|$ is the cardinal number of a set. The rationale for this is clear
from\ Figure 1 and (\ref{r:order}), (\ref{r:auto}).

Without loss of generality we construct $c_{0}(0,S)$ for some $S>0$. For
each $t\in \lbrack 0,S]$ we choose an open interval $I_{t}\subset \mathbb{R}$
containing $t$ in the following way:

(i)\ If $w_{0}(t)\not=0$, then $I_{t}$ is such that for all $s\in I_{t}$, $%
w_{0}(s)\not=0$.

(ii) If $w_{0}(t)=0$, a regular zero, then $I_{t}$ is such that for all $%
s\in I_{t}$, $w_{i-1}(s)\not=0$, $w_{i+1}(s)\not=0$.

(iii) If $w_{0}(t)=0$ is a singular zero of Type I, degree odd, zero at $%
-l,-l+1,...,l$ (i.e centered at $0$), we choose $I_{t}$ so that for all $%
s\in I_{t}$, $s\not=t$, we have $u_{j}(s)\not=0$ for $j=-l,...,-1,1,...,l$.

(iv)\ If If $w_{0}(t)=0$ is a singular zero in all other cases, $I_{t}$ is
such that for all $s\in I_{t}$, $s\not=t$, we have $w_{0}(s)\not=0$.

We constructed the neighborhoods $I_{t}$ in such a way that for each time $%
s\in I_{t}$, there is at most one zero which may be crossing $i=0$ (the
singular zero cases follow from all the discussion above in this section.)

We first define $c(s,s^{\prime })$ for $w(s),w(s^{\prime }),s<s^{\prime }$
in the same interval $I_{t}$, by cases (i)-(iv) above:%
\begin{equation*}
\begin{array}{ll}
\text{Case (i)} & c_{0}(s,s^{\prime })=0, \\ 
\text{Case (ii)} & c_{0}(s,s^{\prime })=z_{0}(w(s^{\prime }))-z_{0}(w(s)),
\\ 
\text{Case (iii) } & c_{0}(s,s^{\prime })=\left\{ 
\begin{array}{ll}
z_{0,l}(w(s^{\prime }))+l-z_{0,l}(w(s)) & s^{\prime }>t\geq s \\ 
0 & \text{otherwise,}%
\end{array}%
\right.  \\ 
\text{Case (iv)} & c_{0}(s,s^{\prime })=0\text{.}%
\end{array}%
\end{equation*}%
Clearly $c_{0}(s,s^{\prime })\in \{-1,0,1\}$. Now assume $I_{t_{0}}$,$%
I_{t_{1}},...I_{t_{j}}$ is a finite subcover of $[0,S]$, and $%
0=s_{0}<s_{1}<...<s_{j^{\ast }}=S$ such that $s_{i-1},s_{i}$ are in the same
open interval $I_{t_{i}}$. We define%
\begin{equation*}
c_{0}(0,S)=\sum_{i=0}^{j^{\ast }-1}c_{0}(s_{i-1},s_{i})\text{.}
\end{equation*}%
It is now straightforward to verify that the definition is independent of
the choice of $t_{i}$, and that $c_{i},d_{i}$ satisfy (A1)-(A4).

\subsection{\textbf{Proof of Lemma \protect\ref{p:monotonicity}}.}

We slightly adapt \cite{Baesens05}, Proposition 2.1. Let $C$ be the uniform
bound on $|c_{i}(t)|$, and let $\overline{w}_{i}(t)=\exp (Ct)w_{i}(t)$. Then 
$\overline{w}_{i}(t)$ is a solution of an equation of type (\ref{r:system})
with $\overline{a}_{i}=a_{i},\overline{b}_{i}=b_{i},\overline{c}_{i}=c_{i}+C$%
, thus $\overline{a}_{i},\overline{b}_{i},\overline{c}_{i}\geq 0$ satisfying
(\ref{r:conditions}). Write it shortly 
\begin{equation}
d\overline{w}/dt=\overline{A}(t)\overline{w}(t),  \label{r:eqeq}
\end{equation}
where $\overline{A}(t)$ is an infinite matrix. By standard results on
existence of solutions on Banach spaces, the solution of (\ref{r:eqeq})
exists locally in any of the norms (\ref{r:norm}) and is the limit of Picard
iterations%
\begin{equation*}
\overline{w}^{(n)}(t)=\overline{w}(0)+\int_{0}^{t}\overline{A}(s)\overline{w}%
^{(n-1)}(s)ds,
\end{equation*}%
$\overline{w}^{(0)}(t)=\overline{w}(0)$. As $\overline{A}(s)$ has
non-negative elements, we easily see that $\overline{w}(0)\geq 0$ implies $%
\overline{w}(t)\geq 0$ for $t\geq 0$, and $w_{i}(t)$ and $\overline{w}%
_{i}(t) $ have the same sign.

\section{Appendix 2:\ On zeroes of degree $\infty $}

This Appendix is dedicated to proofs of\ Lemmas \ref{l:infinite} and \ref%
{l:der2}. Assume without loss of generality that the zero of degree $\infty $
of $w(0)=u(0)-v(0)$ is at $i=1$, i.e. $w_{0}(0)\not=0$, $w_{i}(0)=0$ for all 
$i\geq 1$. We first consider the case when all $u_{i}$ are strictly monotone
on either sides of $t=0$. The proof is somewhat combinatorial:\ we consider
different alternatives for the signs of $du_{j}/dt.$

\begin{lemma}
\label{l:three}Assume all as in Lemma \ref{l:infinite}. Assume also that for
some $i_{0}\geq 1$, there exist $\delta _{j}>0$ so that $du_{j}/dt$ is not
zero on $(-\delta _{j},0)$ and $(0,\delta _{j})$ for all $j\geq i_{0}$. Then
the conclusion of Lemma \ref{l:infinite} holds.
\end{lemma}

\begin{proof}
First note that the conclusion of Lemma \ref{l:infitiy} holds for $%
w(0)=u(0)-v(0)$, and recall that $u_{j}(0)=v_{j}(0)$ for $j\geq 1$.

\textbf{Case 1}. Assume that for some $j_{1}\geq i_{0}$, the signs of $%
du_{j_{1}}/dt,...,du_{j_{1}+3}/dt$ are $+,+,+,+$ (analogously we show $%
-,-,-,-$) for $t\in (-\delta _{j_{1}},0),...,(-\delta _{j_{1}+3},0)$. Now
because of (\ref{r:auto}), for some $\varepsilon >0\,\ $small enough, for
any $t_{0}\in (-\varepsilon ,0)$, $u_{j_{1}}(t_{0})>v_{j_{1}}(t_{0})$, $%
u_{j_{1}+1}(t_{0})<v_{j_{1}+1}(t_{0})$, $%
u_{j_{1}+2}(t_{0})>v_{j_{1}+2}(t_{0})$; or $%
u_{j_{1}+1}(t_{0})>v_{j_{1}+1}(t_{0})$, $%
u_{j_{1}+2}(t_{0})<v_{j_{1}+2}(t_{0})$, $%
u_{j_{1}+3}(t_{0})>v_{j_{1}+3}(t_{0})$. Consider without loss of generality
the first possibility and fix some $0>t_{0}>-\varepsilon $. Now find $%
0>t_{1}>t_{0}$ so that $u_{j_{1}+1}(t_{1})=v_{j_{1}+1}(t_{0})$. As by
assumptions $u_{j_{1}}(t_{1})>v_{j_{1}}(t_{0})$ and $%
u_{j_{1}+2}(t_{1})>v_{j_{1}+2}(t_{0})$, we see that $u(t_{1})-v(t_{0})$ has
a singular zero of degree $1$ at $j_{1}+1$.

\textbf{Case 2}$.$ Assume that for some $i_{0}\leq j_{1}<j_{2}<j_{3}<j_{4}$
(not necessarily consecutive), the signs of $du_{j_{1}}/dt,...,du_{j_{4}}/dt$
alternate on $(0,\delta _{j_{1}}),...,(0,\delta _{j_{4}})$, say $-,+,-,+\,\ $%
(analogously $-,+,-,+$). By (\ref{r:auto}), either $u_{j}(t)>v_{j}(t)$ for
all $j\geq i_{0}$ and all $t\in (0,\delta _{j})$, or $u_{j}(t)<v_{j}(t)$,
consider again only the first possibility. Fix $\,0<t_{0}<\varepsilon $ for $%
\varepsilon $ small enough, and then we can find $0<t_{2}<t_{0}$ so that $%
u_{j_{2}}(t_{2})=v_{j_{2}}(t_{0})$. Thus we can find $t_{2}\leq t_{1}<t_{0}$%
, as the smallest such time such that $u_{j}(t_{1})\geq v_{j}(t_{0})$ for
all $j=j_{1}+1,...,j_{3}-1_{{}}$, and then for $t=t_{1}$ at least one of the
inequalities is $=$. As by assumptions $u_{j_{1}}(t_{1})>v_{j_{1}}(t_{0})$, $%
u_{j_{3}}(t_{1})>v_{j_{3}}(t_{0})$, we see that $u(t_{1})-v(t_{0})$ has a
singular zero of finite degree.

\textbf{Case 3}. Assume that for some $j_{1}\geq i_{0}$, $k\geq 2$ the signs
of $du_{j_{1}}/dt,...,du_{j_{1}+k}/dt$ have the signs $-,+,+,...,+,-$ on $%
(-\delta _{j},0)$ and the signs $\,-,-,...,-\ $on $(0,\delta _{j})$ (or the
alternate sign case). Fix $0<t_{0}<\varepsilon $ for $\varepsilon $ small
enough, and find $-\varepsilon <t_{1}<0$ as the largest $t<0$ so that $%
u_{j}(t)\geq u_{j}(t_{0})$ for all $j=j_{1}+1,...,j_{1}+k-1$, and then for $%
\varepsilon $ small enough at least one of inequalities is $=$. As by
assumptions, $u_{j_{1}}(t_{1})>u_{j_{1}}(t_{0})$, $%
u_{j_{1}+k}(t_{1})>u_{j_{1}+k}(t_{0})$, $u(t_{1})-u(t_{0})$ has a singular
zero of a finite degree.

Now it is straightforward to check that one of the cases 1-3 (including the
alternate possibilities) must hold.
\end{proof}

\begin{proof}[\textbf{Proof of Lemma \protect\ref{l:infinite}.}]
We consider separately the DC and AC\ case. In the DC case, we use the fact
that all the results on the zero-set apply to $w=du/dt$. As this does not
necessarily hold in the AC case, we apply\ real analyticity of the solution.

Consider the DC\ case. If $du_{j}(0)/dt\not=0$, then $u_{j}$ is strongly
monotone in a neighborhood of $t=0$. Also if $du_{j}(0)/dt=0$, by Lemmas \ref%
{l:one}, \ref{l:two} and \ref{l:infitiy}, $du_{j}/dt=d_{j}t^{k}+o(t^{k})$
for some $d_{j}\not=0$ and an integer $k\geq 1$ in all cases except possibly
when $(j,0)$ is a part of a zero of Type II of an even degree of $w=du(0)/dt$%
. If $w=du(0)/dt$ has only finitely many zeroes of Type II, even degree, on $%
(i,0)$ for $i\geq 1$, we see from this discussion that for some $i\geq i_{0}$
the assumptions of Lemma \ref{l:three} hold and the proof is completed. Now
assume the contrary, and consider a zero of Type II, even degree of $%
w=du(0)/dt$, say at $(j+1,0),...,(j+k,0)$. Then $du_{j}(0)/dt$ and $%
du_{j+k+1}(0)/dt\,$have alternate, non-zero signs. As there are infinitely
many such zeroes, we take any three and obtain the Case 2 in the proof of
Lemma \ref{l:three}.

Consider now the AC case, and assume without loss of generality $F(t)$ in (%
\ref{r:FK}) is not constant (if it is, this reduces to the DC\ case). By
assumptions and Lemma \ref{l:existence}, $u_{j}(t)$ is real analytic on
every open subset of the set of times for which $u(t)\in \mathcal{X}$. Then
we have either 
\begin{equation}
du_{j}/dt=d_{j}t^{k}+o(t^{k})\text{, }k\geq 0  \label{r:kk}
\end{equation}%
or $du_{j}/dt=0$ in a neighborhood of $t=0$ (see e.g. \cite{Krantz02},
Remark 1.2.13). If for all $j\geq i_{0}$, the first possibility holds, we
can apply Lemma \ref{l:three} and complete the proof. Assume the contrary,
i.e. existence of infinitely many $j\geq 1$ so that $du_{j}/dt$ locally
vanishes around $t=0$, and discuss all the possibilities. 1) If there are
three or more consecutive vanishing $du_{j}/dt$, inserting it into the
derivative of (\ref{r:FK}) implies that $f(t)$ is locally constant, thus as
real analytic and periodic, constant everywhere, which is a contradiction.
2) Assume there are exactly two consecutive vanishing $du_{j+1}/dt$, $%
du_{j+2}/dt$, for some $t\in (-\varepsilon ,\varepsilon )$, where $%
\varepsilon >0$ is small enough such that $du_{j}/dt,du_{j+3}/dt$ have no
zeroes on $(0,\varepsilon )$ (this is possible as they satisfy (\ref{r:kk}%
)). Then for any $0<t_{0}<t_{1}<\varepsilon $, clearly $u(t_{0})-u(t_{1})$
has a zero of degree 2 at $j+1,j+2$. 3)\ Assume there is a vanishing $%
du_{j+1}/dt$ on $t\in (-\varepsilon ,\varepsilon )$ such that $du_{j}/dt$
and $du_{j+2}/dt$ have no zeroes for $t\in (0,\varepsilon )$ and have the
same sign. Then analogously for any $0<t_{0}<t_{1}<\varepsilon $, $%
u(t_{0})-u(t_{1})$ has a singular zero of degree 1 at $j+1$. 4) Finally,
consider infinitely many vanishing $du_{j}/dt$ so that the signs of $%
du_{j-1}/dt$, $du_{j+1}/dt$ are not zero and alternate for $t\in (0,\delta
_{j})$. But this is the case 2 from the proof of Lemma \ref{l:three}. As we
have exhausted all the possibilities, the proof is completed.
\end{proof}

\begin{proof}[\textbf{Proof of Lemma \protect\ref{l:der2}.}]
We can apply Lemma \ref{l:infitiy} to $w=du/dt$ and immediately see that the
Case 3 of the proof of Lemma \ref{l:three} holds. So the conclusion of the
case 3 must hold, which completes the proof.
\end{proof}

\end{document}